\documentclass[11pt,letterpaper]{article}

\usepackage{ntheorem}

\usepackage[utf8]{inputenc}
\usepackage{comment}
\usepackage{microtype}
\usepackage{graphicx}
\usepackage{subfig}
\usepackage{dsfont}
\usepackage{todonotes}

\usepackage{booktabs} 
\usepackage{scalerel}
\usepackage[margin=1in]{geometry}

\usepackage{amssymb}

\usepackage{amsmath}
\usepackage{mathtools}

\usepackage{nicefrac}

\usepackage{algorithmic}
\usepackage{algorithm}
\usepackage[shortlabels]{enumitem}
\usepackage{url}
\usepackage{float}
\usepackage{pifont}
\usepackage{hhline}
\usepackage{multirow}
\usepackage{graphicx,wrapfig}

\def\grad{\nabla}
\def\noi{\noindent}

\def\bw{\mathbf{w}}
\def\bx{\mathbf{x}}  

\def\bC{\mathbf{C}}

\def\bL{\mathbf{L}}

\def\cD{\mathcal{D}}

\def\cG{\mathcal{G}}

\def\cL{\mathcal{L}}

\def\cN{\mathcal{N}}
\def\cO{\mathcal{O}}
\def\cP{\mathcal{P}}

\def\cR{\mathcal{R}}

\def\cU{\mathcal{U}}

\def\cX{\mathcal{X}}
\def\cY{\mathcal{Y}}
\def\cZ{\mathcal{Z}}

\def\mE{\mathbb{E}}

\def\smskip{\smallskip}

\def\texitem#1{\par\smskip\noindent\hangindent 25pt
               \hbox to 25pt {\hss #1 ~}\ignorespaces}

\def\norm#1{\left\|#1\right\|}

\def\fprod#1{\left\langle#1\right\rangle}

\newcommand{\BEAS}{\begin{eqnarray*}}
\newcommand{\EEAS}{\end{eqnarray*}}
\newcommand{\BEA}{\begin{eqnarray}}
\newcommand{\EEA}{\end{eqnarray}}
\newcommand{\BEQ}{\begin{eqnarray}}
\newcommand{\EEQ}{\end{eqnarray}}
\newcommand{\BIT}{\begin{itemize}}
\newcommand{\EIT}{\end{itemize}}
\newcommand{\BNUM}{\begin{enumerate}}
\newcommand{\ENUM}{\end{enumerate}}

\newcommand{\BA}{\begin{array}}
\newcommand{\EA}{\end{array}}

\newcommand{\reals}{\mathbb{R}}

\DeclareMathOperator*{\argmin}{\arg\!\min}
\DeclareMathOperator*{\argmax}{\arg\!\max}

\newif\ifpagenumbering
\pagenumberingtrue

\pagenumberingfalse

\def\btheta{\boldsymbol{\theta}}

\newcommand{\ey}[1]{\textcolor{black}{#1}}
\newcommand{\eyh}[1]{\textcolor{black}{#1}}
\newcommand{\mma}[1]{\textcolor{black}{#1}}

\newtheorem{assumption}{Assumption}
\newtheorem{definition}{Definition}
\newtheorem{theorem}{Theorem}
\newtheorem{lemma}[theorem]{Lemma}

\newtheorem{corollary}[theorem]{Corollary}
\theoremstyle{remark}
\newtheorem{remark}{Remark}

\numberwithin{assumption}{section}
\numberwithin{definition}{section}
\numberwithin{theorem}{section}
\numberwithin{remark}{section}

\newcommand{\qed}{\mbox{}\hspace*{\fill}\nolinebreak\mbox{$\rule{0.7em}{0.7em}$}}
\newenvironment{proof}{\par{\noi \bf Proof: }}{\(\qed\) \par}

\usepackage{hyperref}
\hypersetup{colorlinks,citecolor=blue,linktocpage,breaklinks=true}

\begin{document}

\title{Single-loop Projection-free and Projected Gradient-based Algorithms for Nonconvex-concave Saddle Point Problems with Bilevel Structure}

\author{Mohammad Mahdi Ahmadi\thanks{Department of Systems and Industrial Engineering, The University of Arizona, Tucson, AZ, USA \qquad\{ahmadi@arizona.edu\}}
\quad Erfan Yazdandoost Hamedani\thanks{Department of Systems and Industrial Engineering, The University of Arizona, Tucson, AZ, USA \qquad\{erfany@arizona.edu\}}
\vspace{5mm}
}



\maketitle

\begin{abstract} 

In this paper, we explore a broad class of constrained saddle point problems with a bilevel structure, wherein the upper-level objective function is nonconvex-concave and smooth over compact and convex constraint sets, subject to a strongly convex lower-level objective function. This class of problems finds wide applicability in machine learning, encompassing robust multi-task learning, adversarial learning, and robust meta-learning. Our study extends the current literature in two main directions: (i) We consider a more general setting where the upper-level function is not necessarily strongly concave or linear in the maximization variable. (ii) While existing methods for solving saddle point problems with a bilevel structure are projected-based algorithms, we propose a one-sided projection-free method employing a linear minimization oracle. Specifically, by utilizing regularization and nested approximation techniques, we introduce a novel single-loop one-sided projection-free algorithm, requiring $\cO(\epsilon^{-4})$ iterations to attain an $\epsilon$-stationary solution, moreover, when the objective function in the upper-level is linear in the maximization component, our result improve to $\cO(\epsilon^{-3})$. Subsequently, we develop an efficient single-loop fully projected gradient-based algorithm capable of achieving an $\epsilon$-stationary solution within $\cO(\epsilon^{-5})$ iterations. This result improves to $\cO(\epsilon^{-4})$ when the upper-level objective function is strongly concave in the maximization component.
 Finally, we tested our proposed methods against the state-of-the-art algorithms for solving a robust multi-task regression problem to showcase the superiority of our algorithms.
\end{abstract}


\section{Introduction}\label{sec:intro}
\ey{Bilevel optimization is an important class of optimization problems where one optimization problem is nested within another. Many learning and inference problems take this form with two levels of hierarchy. On the other hand, saddle point (SP) problems, also known as min-max optimization, have received a lot of attention recently due to their versatile formulation covering a broad class of optimization problems, e.g., convex optimization with nonlinear conic constraints, which itself includes linear programming (LP), quadratic programming (QP), quadratically constrained quadratic programming
(QCQP), second-order cone programming (SOCP), and semidefinite programming (SDP) as its subclasses.}
From a broader perspective, in this paper, we study a saddle point problem with a bilevel structure (SPB) of the following form
\begin{align}\label{eq:main-prob}
    & \min_{x \in \mathcal{X}} \max_{y \in \mathcal{Y}}~ \mathcal L(x,y)\triangleq  \Phi(x,\theta^*(x),y) \\ 
    & \text{s.t.}~\theta^{*}(x) \in \argmin_{\theta} g(x, \theta) \nonumber
\end{align}
where \ey{$\Phi:\reals^n\times\reals^m\times \reals^d \to \reals$ is a continuously differentiable function such that $\Phi(x,\theta,\cdot)$ is concave for any given $x\in \cX$ and $\theta\in \reals^m$; $\cX\subset \reals^n$ and $\cY\subset\reals^d$ are convex and compact sets; moreover, $g: \reals^n \times \mathbb{R}^m \to \mathbb{R}$ is a continuously differentiable and strongly convex in $\theta$. This problem simplifies to bilevel optimization when the upper-level objective function $\Phi$ is separable in the maximization variable and it encompasses saddle point problems when there is no lower-level problem, i.e., $\Phi$ is independent of the second argument.}
The problem in \eqref{eq:main-prob} has a wide range of applications including robust multi-task learning \cite{michel2021balancing}, robust meta-learning \cite{gu2021nonconvex}, adversarial learning \cite{zhang2022revisiting}, representation learning and hyperparameter optimization \cite{gu2022min}, multi-task deep AUC maximization, and multi-task deep partial AUC maximization \cite{hu2022multi}.  
\subsection{Literature Review}
In this section, we will begin by briefly discussing research in bilevel optimization and SP problems. Then, we will touch upon existing methods for solving SPB problems.
\paragraph{Bilevel Optimization.} 
Researchers have proposed numerous algorithms for solving bilevel optimization problems since its initial introduction by Bracken and McGill in 1973 \cite{bracken1973mathematical}. 
In recent years, there has been growing research focus on applying \eyh{gradient-based methods for solving} bilevel problems, incorporating techniques such as implicit differentiation \cite{domke2012generic,ghadimi2018approximation,gould2016differentiating} and iterative differentiation \cite{franceschi2018bilevel,grazzi2020iteration,maclaurin2015gradient,shaban2019truncated}. 
\eyh{The majority of existing studies proving convergence rate guarantees have focused on unconstrained scenario 
\cite{chen2022single,ji2020provably,khanduri2021near,kwon2023fully,li2022fully,yang2023accelerating}, while fewer studies are considering constrained bilevel optimization problems, particularly those with constraints in the upper-level. In particular, researchers have proposed several methods, including projection-based \cite{ghadimi2018approximation,hong2023two} and projection-free \cite{abolfazli2023inexact,akhtar2022projection} algorithms for various settings to achieve complexity guarantees matching the single-level counterpart.} 

\paragraph{Saddle Point (SP) Problem.} In the context of saddle point (SP) problems, there have been many studies on the convex-concave setting \cite{chambolle2016ergodic,hamedani2021primal,nemirovski2004prox,ouyang2015accelerated}, while there is an increasing number of studies concentrating on developing algorithms for nonconvex-concave setting \cite{boroun2024projection,lin2020gradient,ostrovskii2021efficient,zhang2020single}. In the nonconvex-concave setting, \eyh{various approaches including smoothing \cite{zhang2020single} and acceleration \cite{ostrovskii2021efficient} techniques have been porposed}.
Most of these methods are based on solving regularized subproblems inexactly, leading to a multi-loop algorithm that can present implementation challenges. To overcome such complexities, there has been a growing interest in single-loop primal-dual methods. 
In particular, the authors in \cite{xu2023unified} introduce a unified single-loop alternating gradient projection (AGP) algorithm designed for solving smooth nonconvex-(strongly) concave problems. Their algorithm can find an $\epsilon$-stationary solution with a complexity of $\mathcal{O}(\epsilon^{-2})$ and $\mathcal{O}(\epsilon^{-4})$ under the nonconvex-strongly concave and nonconvex-concave settings, respectively. In the optimization literature, limitations in projection-based techniques have prompted researchers to explore efficient projection-free approaches \cite{frank1956algorithm,kolmogorov2021one,lan2013complexity,lan2017conditional}. In the context of saddle point (SP) problems, a limited number of studies have been conducted on projection-free SP problems \cite{boroun2024projection,chen2020efficient,gidel2017frank}. In particular, authors in \cite{boroun2024projection} propose fully and one-sided projection-free algorithms designed to solve an SP problem with a smooth nonconvex-concave objective function. Their proposed method can achieve an $\epsilon$-stationary solution within $\mathcal{O}(\epsilon^{-6})$ and $\mathcal{O}(\epsilon^{-4})$ iterations for fully and one-sided projection-free methods, respectively.
\paragraph{Saddle Point Problem with Bilevel Structure.} Within the existing literature, the focus on the SPB problem, as presented in problem \eqref{eq:main-prob}, is very limited. 
\eyh{In \cite{hu2022multi} two simple, single-loop, single-timescale stochastic methods with randomized block-sampling have been proposed to solve \eqref{eq:main-prob}. Considering an unconstrained setting and assuming that $\Phi$ is smooth, $\Phi(x,\theta,\cdot)$ is strongly concave for any $x,\theta$, and $g(x,\cdot)$ is strongly convex for any $x$, they show that their proposed algorithm converges to an $\epsilon$-stationary point with a complexity of $\mathcal{O}(\epsilon^{-4})$ under a general unbiased stochastic oracle model. In another study, \cite{gu2021nonconvex} motivated by a robust meta-learning problem, consider a bilevel optimization with minimax upper-level objective function, i.e., $\min_{x\in\cX}\max_{i=1}^n f_i(x,\theta_i^*(x))$. This problem can be reformulated as a special case of \eqref{eq:main-prob} by letting $\Phi(x,\theta,y)=\sum_{i=1}^ny_if_i(x,\theta_i)$, i.e., $\Phi$ is assumed to be linear in $y$. Assuming that $\Phi$ is smooth, their proposed algorithm can achieve a complexity of $\cO(\epsilon^{-5})$.} 
\eyh{In recent studies \cite{gu2021nonconvex,gu2022min}, a minimax multi-objective bilevel optimization problem with applications in robust machine learning is explored. They propose a single-loop Multi-Objective Robust Bilevel Two-timescale optimization (MORBiT) algorithm for solving \eqref{eq:main-prob} in a deterministic regime when $\Phi$ is linear in $y$.} Their algorithm converges to an $\epsilon$-stationary solution at a rate of $\cO(\epsilon^{-5})$.

\subsection{Contribution}
\eyh{In this paper, we consider a broad class of constrained SP problems with a bilevel structure, where the upper-level objective function is smooth over compact and convex constraint sets, subject to a strongly convex lower-level objective function. While the existing methods for solving \eqref{eq:main-prob} focus on special cases of this problem when either $\Phi$ is linear or strongly concave in $y$, we consider a more general setting where the $\Phi(x,\theta,\cdot)$ is a concave function for any $x,\theta$. Utilizing regularization and nested approximation, we propose novel single-loop inexact bilevel primal-dual algorithms featuring only two matrix-vector products at each iteration. To contend with the challenge of projection onto the constraint set $\cX$, we propose a one-sided projection-free method utilizing a linear minimization oracle.} 
\eyh{Our main theoretical guarantees for the proposed methods are summarized as follows:
\begin{itemize}
    \item The proposed Inexact Bilevel Regularized Primal-dual One-sided Projection-free (i-BRPD:OPF) method achieves an $\epsilon$-stationary solution within $\mathcal{O}(\epsilon^{-4})$ iterations. This complexity result improves to $\cO(\epsilon^{-3})$ when $\Phi$ is linear in $y$. 
    \item The proposed Inexact Bilevel Regularized Primal-dual Fully Projected (i-BRPD:FP) method requires $\mathcal{O}(\epsilon^{-5})$ iterations to find an $\epsilon$-stationary point. \mma{This complexity result improves to $\cO(\epsilon^{-4})$ when $\Phi$ is strongly concave in $y$}. 
\end{itemize}
To the best of our knowledge, the proposed methods and their convergence guarantee are the first to appear in the literature for the considered setting.}
\subsection{Motivating Examples}\label{subsec:examples}


\ey{The SPB problem in \eqref{eq:main-prob} presents a versatile formulation covering various applications arising} in machine learning such as adversarial learning \cite{zhang2022revisiting}, representation learning and hyperparameter optimization \cite{gu2022min}, robust multi-task learning \cite{michel2021balancing}, robust meta-learning \cite{gu2021nonconvex}, and multi-task deep AUC maximization \cite{hu2022multi}. Next, we describe two examples in detail.

\textbf{Robust Multi-task Learning:} 
\ey{Multi-task learning (MTL) is a machine learning approach that focuses on jointly learning multiple related tasks \cite{caruana1997multitask}. The main objective is to share knowledge across these tasks, allowing them to benefit from each other's insights. More specifically, consider a set of $T$ tasks denoted by $\{\mathcal{T}_i\}_{i=1}^T$, each equipped with a training dataset $\{\mathcal{D}_i^{tr}\}_{i=1}^T$ containing features residing in the same space. In the context of these tasks, we employ task-specific loss functions $\{\ell_i(x,\theta_i,\mathcal{D}_i^{tr})\}_{i=1}^T$ where $\theta_i\in \mathbb{R}^m$ represents task-specific model parameters, and $x$ is the shared parameter. Although many formulations and loss functions have been proposed in the MTL literature, they generally yield a weighted empirical risk minimization $\min_{\theta=[\theta_i]_{i=1}^T}\sum_{i=1}^T y_i\ell_i(x,\theta_i;\mathcal D_i^{tr})$, where the weights $\{y_i\}_{i=1}^T$ play a crucial role in the performance of the model as more favor is given to the tasks with higher weights. Therefore, it is natural to consider learning the weights through a robust optimization formulation \cite{chen2022gradient,shu2019meta} and learning the shared parameter $x$ through the lens of bilevel optimization \cite{gu2021nonconvex}. This problem can be uniformly formulated as
\begin{align*}
	&\min_{x \in \mathcal{X}}\max_{y \in \cU} \sum_{i=1}^{T} y_i \ell_i(x, \theta_i^*(x); \cD_i^{val}) \\
&\text{s.t.} \quad \theta_i^*(x) = \argmin_{\theta_i \in \mathbb{R}^{m}}\ell_i(x,\theta_i; \cD_i^{tr}) + \cR(\theta_i), \quad i \in \{1,\hdots,T\}, \nonumber
\end{align*}
where $\cU\subset\reals^T$ is the weight constraint set, e.g., $\cU=\Delta_T$ and $\Delta_T$ denotes the simplex set defined as $\Delta_T \triangleq \{y \in \mathbb{R}^T_+ | \sum_{i=1}^T y_i= 1 \}$, $\cD_i^{val}$ is the validation dataset for task $i$, and $\cR:\reals^m\to\reals$ is a regularization function, e.g., $\cR(\theta_i)=\frac{\rho}{2}\norm{\theta_i}^2$ for some $\rho>0$.}

\textbf{Adversarial Learning:} 
\ey{As safeguarding machine learning models against adversarial attacks becomes increasingly vital,  adversarial learning emerges as a pivotal strategy in fortifying these models. By exposing models to adversarial examples during training, adversarial training (AT) \cite{goodfellow2014explaining,madry2017towards} significantly enhances their robustness against potential attacks. }
Recently, in \cite{zhang2022revisiting}, the authors propose a bilevel formulation of the AT, wherein they aim to minimize the expectation of the training loss function at the upper-level formulated as $\min_{x\in\cX}\mE_{\cD^{tr}\sim\cP}[\ell(x,\theta^*(x);\cD^{tr})]$ while minimizing the attack loss function at the lower-level formulated as $\theta^*(x)\in\argmin_{\theta}\ell_{atk}(x,\theta)$. In this formulation, the upper-level involves an expectation that assumes knowledge of the data distribution. However, in real-world applications, such knowledge may not be available. To address this, one standard approach is to utilize the distributionally robust optimization (DRO) \cite{namkoong2016stochastic} technique. Instead of relying on a specific assumed distribution, DRO considers a set of possible distributions, represented by an uncertainty set. 
Therefore, \ey{integrating a DRO formulation within bilevel AT we can formulate the following robust adversarial learning problem:} 
\begin{align*}
    &\min_{x \in \mathcal{X}} \max_{p=[p_i]_{i=1}^N \in \mathcal{U}} ~\sum_{i=1}^N p_i {\ell}(x, \theta^*(x);\cD^{tr}_i) \\
&\text{s.t.} \quad \theta^*(x) \in \argmin_{\theta} \ell_{atk}(x, \theta) \nonumber,
\end{align*}
where 
$\mathcal{U}\subset\reals^N$ denotes the uncertainty set. \ey{Furthermore, when the uncertainty set is represented by a divergence measure, e.g., $\cU=\{p\in\Delta_N\mid V(p,\frac{1}{N}\mathbf{1}_N)\leq \rho\}$ for some divergence metric $V$ and $\rho>0$, then projection onto $\cU$ is computationally expensive. Therefore, using the Lagrangian method and introducing variable $\lambda\in\reals_+$ we can relax the nonlinear constraint in $\cU$ to equivalently represent the problem as 
\begin{align}\label{eq:ad-general}
    &\min_{\substack{x \in \mathcal{X}\\ \lambda\geq 0}} \max_{p\in\Delta_N} ~\sum_{i=1}^N p_i {\ell}(x, \theta^*(x);\cD^{tr}_i)]-\lambda\big(V(p,\tfrac{1}{N}\mathbf{1}_N)-\rho\big) \\
&\text{s.t.} \quad \theta^*(x) \in \argmin_{\theta} \ell_{atk}(x, \theta) \nonumber.
\end{align}
Indeed, one can easily verify that \eqref{eq:ad-general} is a special case of \eqref{eq:main-prob}. Note that the upper-level objective function in \eqref{eq:ad-general} is neither linear in the maximization component $p$ nor strongly concave.}
\section{Preliminaries}\label{sec:pre}
\subsection{Assumptions and Definitions}
In this section, we state the necessary assumptions required for the upper-level min-max objective function $\Phi$, the lower-level objective function $g$, and the constraint sets $\cX$ and $\cY$. Moreover, we provide definitions essential for analyzing our proposed methods throughout the paper. 

The upper-level objective function $\Phi$ is smooth and concave in $y$.
\begin{assumption}[Upper-level min-max objective function]\label{assump:grad-xytheta-lip}
    $\Phi:\reals^n\times\reals^m\times\reals^d \to \reals$ is a continuously differentiable function such that for any given $x,\theta$, $\Phi(x,\theta,\cdot)$ is a concave function. Moreover, 
   $\nabla_x \Phi(x,\theta,y)$, $\nabla_\theta \Phi(x,\theta,y)$, and $\nabla_y \Phi(x,\theta,y)$ are Lipschitz continuous w.r.t $(x,\theta,y)$ such that for any $x,\overline{x}\in \reals^n$ and $\theta, \overline{\theta}\in \mathbb{R}^m$, and $y, \overline{y} \in \reals^d$ 
   \begin{enumerate}
    \item $\left\| \nabla_x \Phi(x,\theta,y) - \nabla_x \Phi(\overline{x},\overline{\theta},\overline{y}) \right\| \leq L_{xx}^\Phi \left\| x - \overline{x} \right\| + L_{x\theta}^\Phi \left\| \theta - \overline{\theta}\right\| + L_{xy}^\Phi \left\|y - \overline{y} \right\|$,
 \item $\left\| \nabla_\theta \Phi(x,\theta,y) - \nabla_\theta \Phi(\overline{x},\overline{\theta},\overline{y}) \right\| \leq L_{\theta x}^\Phi \left\| x - \overline{x} \right\| + L_{\theta\theta}^\Phi \left\| \theta - \overline{\theta}\right\|+ L_{\theta y}^\Phi \left\|y - \overline{y} \right\|$,
 \item $\left\| \nabla_y \Phi(x,\theta,y) - \nabla_y \Phi(\overline{x},\overline{\theta},\overline{y}) \right\| \leq L_{yx}^\Phi \left\| x - \overline{x} \right\| + L_{y\theta}^\Phi \left\| \theta - \overline{\theta}\right\| + L_{yy}^\Phi \left\|y-\overline{y} \right\|$.
 \end{enumerate}
\end{assumption}

Moreover, we impose the following common assumptions on the lower-level objective function $g$.
\begin{assumption}[Lower-level objective function]\label{assump:g-conditions}
    $g(x,\theta)$ satisfies the following conditions:
    \begin{enumerate}
    \item For any given $x \in \mathcal{X}, g(x,\cdot)$ is twice continuously differentiable. Moreover, $\nabla_\theta g(\cdot,\cdot)$ is continuously differentiable.
\item For any given $x \in \mathcal{X}, \nabla_\theta g(x, \cdot)$ is Lipschitz continuous with constant $L_g \geq0$. Moreover, for any $\theta \in \mathbb{R}^m, \nabla_\theta g(\cdot,\theta)$ is Lipschitz continuous with constant $C_{\theta x}^g \geq0$.
\item For any $x \in \mathcal{X}, g(x, \cdot)$ is $\mu_g$-strongly convex with modulus $\mu_g>0$.
\item For any given $x \in \mathcal{X}, \nabla^2_{\theta x} g(x, \theta) \in \mathbb{R}^{n\times m}$ and $\nabla^2_{\theta\theta} g(x,\theta)$ are Lipschitz continuous w.r.t $(x,\theta) \in \mathcal{X}\times \mathbb{R}^m$, and with constant $L_{\theta x}^g \geq0$ and $L_{\theta\theta}^g \geq0$, respectively.
\end{enumerate}
\end{assumption}
\begin{remark}\label{remak:hessiang-conditions}
    Considering Assumption \ref{assump:g-conditions}-(2), we can conclude that $\left\|\nabla^2_{\theta x}g(x,\theta) \right\|$ is bounded with constant $C^g_{\theta x}$ for any $(x,\theta) \in \mathcal{X}\times \mathbb{R}^m$.\\
\end{remark}

\begin{assumption}\label{assump:grad-phi-bounded}
    There exists $C_\theta^\Phi \geq0$ such that $\left\| \nabla_\theta \Phi(x, \theta^*(x),y) \right\| \leq C_\theta^\Phi$.
\end{assumption}

\begin{assumption}\label{assump:Dx-compact}
    $\cX\subset\reals^n$ and $\cY\subset \reals^d$ are a convex and compact sets with diameters $D_{\mathcal{X}}$ and $D_\cY$, respectively, i.e., $D_{\mathcal{X}} \triangleq sup _{x,\overline{x}\in \mathcal{X}} \left\| x-\overline{x} \right\|_{\mathcal{X}}$ and $D_{\mathcal{Y}} \triangleq sup _{y,\overline{y}\in \mathcal{Y}} \left\| y-\overline{y} \right\|_{\mathcal{Y}}$.
\end{assumption}
In this paper, to handle the primal constraint set $\mathcal{X}$, we use two different oracles, namely a linear minimization oracle (LMO) to find $s=\argmin_{x\in \mathcal{X}}\fprod{x,\bar x}$ for any given $\bar x\in\mathbb R^n$ and a projection oracle (PO) to find $s=\argmin_{x\in \mathcal{X}}\norm{x-\bar x}$ for any given $\bar x\in\mathbb R^n$. Consequently, to measure the quality of solutions generated by the proposed algorithms, we define two different gap functions. The first gap function will utilize LMO for the primal part and the generalized gradient norm for the dual part while the second one will use the generalized gradient norm for both the primal and dual parts.
\begin{definition}\label{def:gap-func-Alg1}
     The stationary gap function $\mathcal{G}_{\mathcal{X}}: \mathcal{X} \times \mathcal{Y} \rightarrow \mathbb R$ for the primal part of Problem \eqref{eq:main-prob} is defined as $\mathcal{G}_{\mathcal{X}}(x,y) \triangleq \sup_{s\in \cX} \langle \grad_x\cL(x,y), x - s \rangle$. For the dual part of Problem \eqref{eq:main-prob}, the stationary gap function $\mathcal{G}_{\mathcal{Y}}: \mathcal{X} \times \mathcal{Y} \rightarrow \mathbb R$ is defined as $\mathcal{G}_{\mathcal{Y}}(x,y) \triangleq \frac{1}{\sigma}\left\| y - \mathcal{P}_{\mathcal{Y}}(y + \sigma \grad_y\cL(x,y))\right\|$. Moreover, we define $\mathcal{G}_{\mathcal{Z}}(x,y) \triangleq \mathcal{G}_{\mathcal{X}}(x,y) + \mathcal{G}_{\mathcal{Y}}(x,y)$.
\end{definition}
\begin{definition}\label{def:gap-func-Alg2}
      The stationary gap function $\mathcal{G}_{\mathcal{X}}: \mathcal{X} \times \mathcal{Y} \rightarrow \mathbb R$ for the primal part of Problem \eqref{eq:main-prob} is defined as \mma{$\cG_\cX(x,y)\triangleq\frac{1}{\tau} \norm{x-\cP_{\cX}(x-\tau \nabla_x \cL(x,y))}$}. The dual gap function is defined similarly to Definition \ref{def:gap-func-Alg1}. Additionally, we define $\mathcal{G}_{\mathcal{Z}}(x,y) \triangleq \mathcal{G}_{\mathcal{X}}(x,y) + \mathcal{G}_{\mathcal{Y}}(x,y)$ where $\mathcal{G}_{\mathcal{Y}}$ is introduced in Definition \ref{def:gap-func-Alg1}.
\end{definition}
\begin{definition}
    For a given gap function $\mathcal{G}_{\mathcal{Z}}:\mathcal{Z} \rightarrow \mathbb R$, where $\mathcal{Z} \triangleq \mathcal{X} \times \mathcal{Y}$, we call a point $(x,y) \in \mathcal{Z}$ an $\epsilon$-stationary solution for Problem \eqref{eq:main-prob} if $\mathcal{G}_{\mathcal{Z}}(x,y)\leq \epsilon$.
\end{definition}
\begin{remark}
    With a slight abuse of notation, we utilized $\mathcal{G}_{\mathcal{X}}$ for both the Linear Minimization Oracle (LMO) and the Projection Oracle (PO) with respect to the constraint set $\mathcal{X}$. In fact, both $\epsilon$-solutions imply an $\epsilon$-game stationary solution. Specifically, if $\mathcal{G}_{\mathcal{Z}}(x,y)\leq \epsilon$ holds for some $(x,y)\in \mathcal{Z}$, then 
    \begin{align*}
        & 0 \in \partial \mathds{1}_{\mathcal{X}}(x) + \nabla_x \mathcal{L}(x,y) + u, \quad \text{s.t.} \quad u \in \mathcal{B}_{\mathbb{R}^n}(0,\epsilon)\\
        & 0 \in \partial \mathds{1}_{\mathcal{Y}}(y) - \nabla_y \mathcal{L}(x,y) + v, \quad \text{s.t.}\quad  v \in \mathcal{B}_{\mathbb{R}^d}(0,\epsilon),
    \end{align*}%
    where we define $\mathcal{B}_{\mathcal{U}}(\bar u,r) \triangleq \big\{u \in \mathcal{U} \big| \|u-\bar u\|\leq r\big\}$ for a given vector space $\mathcal{U}$ and $\bar u\in\cU$.
\end{remark}

\subsection{Problem Properties}\label{subsec:proposed-method}
Problem \eqref{eq:main-prob} can be viewed as an implicit SP problem, i.e., $\min_{x\in\cX}\max_{y\in\cY}\cL(x,y)$. Therefore, it is natural to first investigate the two implicit partial gradients of the objective function: one corresponding to the minimization variable $x$ and the other to the maximization variable $y$. More specifically, due to the uniqueness of the solution and smoothness of the lower-level objective function one can show that
\begin{subequations}
\begin{align}
    & \nabla_x\mathcal L(x,y) =\nabla_x \Phi(x,\theta^*(x),y) + {\bf J}\theta^*(x)\nabla_\theta\Phi(x,\theta^*(x),y) \label{eq:gradx-L-jacob}\\
    & \nabla_y\mathcal L(x,y) = \nabla_y \Phi(x,\theta^*(x),y), \label{eq:grady-L}
\end{align}
\end{subequations}
where ${\bf J}\theta^*(x) \in \mathbb R^{m \times n}$ is the Jacobian of $\theta^*(x)$. The Jacobian matrix can be characterized based on the implicit function theorem for the lower-level problem by writing its first-order optimality condition, namely $\nabla_\theta g(x,\theta^*(x)) = 0$ for any $x\in\reals^n$. 
By taking the derivative with respect to $x$, 
one can obtain ${\bf J}\theta^*(x) = -\nabla_{\theta x}^2 g(x,\theta^{*}(x)) [\nabla_{\theta\theta}^2 g(x,\theta^{*}(x))]^{-1}$. 
Substituting this relation in \eqref{eq:gradx-L-jacob}, we have
\begin{align*}
    \nabla_x\mathcal L(x,y) =\nabla_x \Phi(x,\theta^*(x),y) -\nabla_{\theta x}^2 g(x,\theta^{*}(x)) [\nabla_{\theta\theta}^2 g(x,\theta^{*}(x))]^{-1} \nabla_\theta\Phi(x,\theta^*(x),y).
\end{align*}
which can be rewritten as follows
\begin{subequations}
\begin{align}\label{eq:grad-x-L}
    & \nabla_x\mathcal L(x,y) =\nabla_x \Phi(x,\theta^*(x),y) -\nabla_{\theta x}^2 g(x,\theta^{*}(x)) v(x,y), \\
    & \text{where}\quad  v(x,y) = [\nabla_{\theta\theta}^2 g(x,\theta^{*}(x))]^{-1} \nabla_\theta\Phi(x,\theta^*(x),y). \label{eq:v}
\end{align}
\end{subequations}
The above formulation presents several challenges inherent to bilevel optimization problems, notably finding $\theta^*(x)$ and the inversion of the Hessian matrix. Furthermore, our SP structure adds another layer of complexity compared to typical bilevel optimization setups. This complexity stems from the interplay, both explicit and implicit, between the maximization variable $y$, the primal variable $x$, and the lower-level optimal solution $\theta^*(x)$, as represented in equation \eqref{eq:grady-L}. 

Based on the considered setting and assumptions, we first establish several important properties of the problem. In particular, we first establish the Lipschitz continuity of map $v$ in both primal and dual variables. Consequently, we can show the Lipschitz continuity of the lower-level optimal solution and partial derivatives of the implicit objective function $\mathcal L$. These results serve as the cornerstone for the analysis of our proposed methods in the following section.
\begin{lemma}\label{lem:lower-v}
    Suppose Assumptions \ref{assump:grad-xytheta-lip} and \ref{assump:g-conditions} hold. Then for any $x, \overline{x} \in \mathcal{X}$ and $y, \overline{y} \in \mathcal{Y}$, we have that 
    $\left\| v(x,y)-v(\overline{x},\overline{y}) \right\| \leq \bC_{v1}\left\| x -\overline{x} \right\| + \bC_{v2}\left\| y -\overline{y} \right\|$ where $\bC_{v1} \triangleq  \frac{L_{\theta x}^\Phi + L_{\theta \theta}^\Phi\bL_\theta}{\mu_g} + \frac{C_\theta^\Phi L_{\theta\theta}^g}{\mu^2_g} (1+\bL_\theta)$ and $\bC_{v2} \triangleq \frac{1}{\mu_g}L_{\theta y}^\Phi$. 
\end{lemma}
\begin{proof}
    See Appendix \ref{sec:proof-lemma-lip-v} for the proof.
\end{proof}
\begin{lemma}\label{lem:theta-gradxL-lip}
    Suppose Assumptions \ref{assump:grad-xytheta-lip}, \ref{assump:g-conditions} hold. Then for any $x, \overline{x} \in \mathcal{X}$ and $y, \overline{y} \in \mathcal{Y}$, the following results hold.\\
\textbf{(I)} There exists $\bL_\theta \geq0$ such that $\left\| \theta^*(x) - \theta^*(\overline{x}) \right\| \leq \bL_\theta \left\| x - \overline{x} \right\|$ where $\bL_\theta \triangleq \frac{C_{\theta x}^g}{\mu_g}$.\\
\textbf{(II)} There exists $\bL_{\mathcal L_1}, \bL_{\mathcal L_2} \geq0$ such that $\left\| \nabla_x\mathcal L(x,y) - \nabla_x\mathcal L(\overline{x},\overline{y}) \right\| \leq \bL_{\mathcal L_1} \left\| x - \overline{x} \right\| + \bL_{\mathcal L_2} \left\| y - \overline{y} \right\|$ where $\bL_{\mathcal L_1} \triangleq (L_{xx}^\Phi  + L_{x\theta}^\Phi \bL_\theta + C_{\theta x}^g \bC_{v1} + \frac{C_\theta^\Phi}{\mu_g} L_{\theta x}^g (1 + \bL_\theta)), \bL_{\mathcal L_2} \triangleq (L_{xy}^\Phi + C_{\theta x}^g \bC_{v2})$.\\
\textbf{(III)} $\| \nabla_y\mathcal L(x,y) - \nabla_y\mathcal L(\overline{x},\overline{y}) \| \leq (L_{yx}^\Phi + L_{y\theta}^\Phi \bL_\theta) \left\| x - \overline{x} \right\| + L_{yy}^\Phi \left\|y-\overline{y} \right\|$.
\end{lemma}
\begin{proof}
    See Appendix \ref{sec:proof-lemma-basic} for the proof.
\end{proof}

\section{Proposed Methods}\label{sec:main-algorithm}
In this section, we present our proposed algorithms based on an inexact bilevel regularized primal-dual approach for solving Problem \eqref{eq:main-prob}. 
We consider two main approaches to deal with the constraint $\cX$ in the minimization: the first one involves using a Linear Minimization Oracle (LMO), while in the second approach, we use a Projection Oracle (PO) to update the primal step of the upper-level objective function. To update the maximization variable, we consider a gradient ascent-based approach. Moreover, in the lower-level problem, we incorporate one step of gradient-type update to estimate the lower-level optimal solution. As a result, we propose both a one-sided projection-free and a fully projected gradient-based algorithms. Next, we discuss the detailed update of our proposed methods.

\textbf{Primal update:} One common approach to avoid using Hessian inverse in \eqref{eq:v} is to solve the corresponding parametric quadratic programming problem as follows
\begin{align*}
    v(x,y) = \argmin_v \frac{1}{2} v^\top \nabla_{\theta \theta}^2 g(x,\theta^*(x))v - \nabla_\theta \Phi(x, \theta^*(x),y)^\top v.
\end{align*}
The solution to the above quadratic programming problem can be approximated using one step of the gradient descent technique \cite{khanduri2021near,li2022fully}, as follows:
\begin{align*}
    v(x_k,y_k) = v(x_k,y_k) - \eta_k \big(\nabla_{\theta \theta}^2 g(x_k,\theta^*(x_k))v(x_k,y_k) - \nabla_\theta \Phi(x_k,\theta^*(x_k),y_k)\big).
\end{align*}
for some step-size $\eta_k \geq0$. In order to implement an iterative method for solving the SPB problem using first-order information at each iteration $k\geq0$, the exact evaluation of $\theta^*(x_k)$ is typically not feasible. One approach is to substitute $\theta^*(x_k)$ with its estimated solution $\theta_k$, which again can be obtained by applying one step of the gradient descent method with respect to the lower-level objective function to trace the optimal trajectory of the lower-level problem. Subsequently, $v(x_k,y_k)$ can be approximated by an iterate $w_{k+1}$ obtained by taking one step of gradient descent as
\begin{align*}
    w_{k+1} \leftarrow w_k - \eta_k \big(\nabla_{\theta \theta}^2 g(x_k, \theta_k) w_k - \nabla_\theta \Phi(x_k, \theta_k, y_k)\big).
\end{align*}
Next, the directions estimating \eqref{eq:grad-x-L} and \eqref{eq:grady-L} can be obtained as 
\begin{align*}
   & G^x_k =\nabla_x \Phi(x,\theta_k,y) -\nabla_{\theta x}^2 g(x,\theta_k) w_{k+1},\\
   & G^y_k = \nabla_y \Phi(x,\theta_k,y).
\end{align*}
Then, based on the oracle available for the constraint set $\mathcal{X}$, we either employ the Frank-Wolfe (FW)-type update \cite{frank1956algorithm} or projected gradient (PG) update based on the Krasnoselskii-Mann method \cite{shehu2018convergence}. In particular, we compute the next iteration $x_{k+1}$ using FW and PG, respectively, as follows:
\begin{align*}
    & FW: s_k \leftarrow \argmin_{x \in \mathcal{X}} \langle G_k^x,x \rangle, \quad x_{k+1} \leftarrow \gamma_k s_k + (1-\gamma_k)x_k,\\
    & PG: s_k \leftarrow \cP_\cX(x_k-\tau_k G_k^x), \quad x_{k+1} \leftarrow \gamma_k s_k + (1-\gamma_k)x_k.
\end{align*}
for some step-size $\gamma_k \in [0,1]$. Given the update $x_{k+1}$,  by implementing a gradient descent step with respect to the lower-level objective function $g(x_{k+1}, \cdot)$ using a step-size $\alpha>0$, we can generate an iterative sequence $\theta_{k+1}$ to approximate the lower-level optimal solution $\theta^*(x_{k+1})$ as
\begin{align*}
    \theta_{k+1} \leftarrow \theta_k - \alpha \nabla_\theta g(x_{k+1}, \theta_k).
\end{align*}
\textbf{Dual update:} In the context of nonconvex-concave (NC-C) SP problem, many approaches have been introduced recently to mitigate the challenge of nonsmoothness and nonconvexity of $f(x) \triangleq \max_{y \in \mathcal{Y}} \mathcal{L}(x,y)$. One approach is dual smoothing in which we add a regularization term $-\frac{\mu}{2}\|y-y_0\|^2$ for some $y_0 \in \mathcal{Y}$, to ensure that $f$ is smooth \cite{zhao2023primal}. Therefore,  we can write the regularized upper-level objective function for Problem \eqref{eq:main-prob} as follows
\begin{align}\label{eq:minmax-lmu}
    \min_{x \in \mathcal{X}} \max_{y \in \mathcal{Y}} \mathcal{L}_\mu (x,y) \triangleq \mathcal{L}(x,y) - \frac{\mu}{2} \|y-y_0\|^2.
\end{align}
Moreover, suppose $f_\mu: \mathcal{X} \rightarrow \mathbb{R}$ be a function such that $f_\mu (x) \triangleq \max_{y\in \mathcal{Y}} \mathcal{L}(x,y) - \frac{\mu}{2} \left\| y - y_0 \right\|^2$, we also define $y^*_\mu (x) \triangleq \argmin_{y \in \mathcal{Y}} \mathcal{L}(x,y) - \frac{\mu}{2} \left\| y - y_0 \right\|^2$. Here, by controlling $\mu$, we can find $y^*_{\mu}(x)$ close to $y^*(x)$ (i.e., $y^*_{\mu}(x) \approx y^*(x)$). Consequently, we can update the dual variable by performing a step of projected gradient ascent, as follows:
\begin{align*}
    y_{k+1} \leftarrow \mathcal{P}_{\mathcal{Y}} \Big(y_k + \sigma_k \big(G_k^y - \mu (y_k - y_0)\big)\Big).
\end{align*}
Our proposed inexact bilevel regularized primal-dual algorithms, one-sided projection-free (i-BRPD:OPF) and fully projected (i-BRPD:FP), are outlined in Algorithms \ref{i-BRPD:OPF} and \ref{i-BRPD:FP}, respectively.
\begin{algorithm}
\caption{Inexact Bilevel Regularized Primal-dual: One-sided Projection-free (i-BRPD:OPF) Method}\label{i-BRPD:OPF}
\begin{algorithmic}[1]
\STATE \textbf{Input}: $x_0 \in \mathcal{X}$, $y_0 \in \mathcal{Y}$, $\theta_0 \in \mathbb{R}^m$, $\alpha, \mu >0$, $\{\gamma_k, \sigma_k, \eta_k\}_k\subseteq\reals_+$
\STATE
\textbf{Initialization}: $\bw^0 \leftarrow \btheta^0$
\FOR{$k = 0,\dots,K-1$}
\STATE
$w_{k+1} \gets w_k-\eta_k (\grad_{\theta\theta}^2g(x_k,\theta_k)w_k-\grad_\theta\Phi(x_k,\theta_k,y_k))$ 
\STATE
$G_k^x \leftarrow \nabla_x \Phi(x_k,\theta_k,y_k) - \nabla_{\theta x}^2 g(x_k, \theta_k)w_{k+1}$
\STATE
$G_k^y \leftarrow \nabla_y \Phi(x_k, \theta_k, y_k)$
\STATE 
$s_k \leftarrow \argmin_{x \in \mathcal{X}} \langle G_k^x,x \rangle$
\STATE
$x_{k+1} \leftarrow \gamma_k s_k + (1-\gamma_k)x_k$
\STATE
$\theta_{k+1} \leftarrow \theta_k - \alpha \nabla_\theta g(x_{k+1}, \theta_k)$
\STATE
$y_{k+1} \leftarrow \mathcal{P}_{\mathcal{Y}} \Big(y_k + \sigma_k \big(G_k^y - \mu (y_k - y_0)\big)\Big)$
\ENDFOR
\end{algorithmic}
\end{algorithm}

It is noteworthy to mention that our proposed algorithms are designed as single-loop algorithms requiring at most two Hessian vector products.  

\begin{algorithm}
\caption{Inexact Bilevel Regularized Primal-dual: Fully Projected (i-BRPD:FP) Method}\label{i-BRPD:FP}
\begin{algorithmic}[1]
\STATE \textbf{Input}: $x_0 \in \mathcal{X}$, $y_0 \in \mathcal{Y}$, $\theta_0 \in \mathbb{R}^m$, $\alpha, \mu >0$, $\{\gamma_k, \sigma_k, \tau_k, \eta_k\}_k\subseteq\reals_+$
\STATE
\textbf{Initialization}: $\bw^0 \leftarrow \btheta^0$
\FOR{$k = 0,\dots,K-1$}
\STATE
$w_{k+1} \gets w_k-\eta_k (\grad_{\theta\theta}^2g(x_k,\theta_k)w_k-\grad_\theta\Phi(x_k,\theta_k,y_k))$ 
\STATE
$G_k^x \leftarrow \nabla_x \Phi(x_k,\theta_k,y_k) - \nabla_{\theta x}^2 g(x_k, \theta_k)w_{k+1}$
\STATE
$G_k^y \leftarrow \nabla_y \Phi(x_k, \theta_k, y_k)$
\STATE 
$s_k \leftarrow \cP_\cX(x_k-\tau_k G_k^x)$
\STATE
$x_{k+1} \leftarrow \gamma_k s_k + (1-\gamma_k)x_k$
\STATE
$\theta_{k+1} \leftarrow \theta_k - \alpha \nabla_\theta g(x_{k+1}, \theta_k)$
\STATE
$y_{k+1} \leftarrow \mathcal{P}_{\mathcal{Y}} \Big(y_k + \sigma_k \big(G_k^y - \mu (y_k - y_0)\big)\Big)$
\ENDFOR
\end{algorithmic}
\end{algorithm}

\section{Analysis of the Proposed Methods}\label{sec:analysis-methods}
In this section, we first delve into the analysis of the lower-level approximation in Subsection \ref{subsec:lower-level-approx}. Moving on to Subsection \ref{subsec:max-component}, we establish a bound on the error between $y_k$ and $y_\mu^*(x_k)$ to examine the trajectory of the optimal solution in the maximization component. Bringing together these subsections, we elucidate the analysis of gradient estimation in Subsection \ref{subsec:gradients-estimation}.
\subsection{Analysis of the Lower-level Approximation}\label{subsec:lower-level-approx}
The first step is to analyze the lower-level to see how $\theta_k$ is tracking $\theta^*(x_k)$, which is crucial for understanding the convergence behavior of proposed algorithms. This analysis serves to establish a bound on the error between $\theta_k$ and $\theta^*(x_k)$, which is essential for assessing the convergence rate of the algorithm.
\begin{lemma}\label{lem:lower_GD_C}
    Suppose Assumption \ref{assump:g-conditions} holds. Let $\{(x_k,\theta_k)\}_{k\geq0}$ be the sequence generated by Algorithm \ref{i-BRPD:OPF}, such that $\alpha = \frac{2}{\mu_g + L_g}$. Then, for any $k\geq0$
\begin{align*}  
    \left\| \theta_k - \theta^*(x_k) \right\| \leq \beta^k \left\| \theta_0 - \theta^*(x_0) \right\| + \bL_\theta \sum_{i=0}^{k-1} \beta^{k-i} \left\| x_i - x_{i+1} \right\|,
\end{align*}%
where $\beta\triangleq (L_g-\mu_g)/(L_g+\mu_g)$. Moreover, the sequence $\{\theta_k \}_{k\geq0}$ is a bounded sequence.
\end{lemma}
\begin{proof}
We begin the proof by characterizing the one-step progress of the lower-level iterate sequence $\{\btheta_{k}\}_k$. Indeed, at iteration $k$, we aim to approximate $\btheta^*(\bx_{k+1})=\argmin_{\btheta} g(\bx_{k+1},\btheta)$. According to the update of $\btheta_{k+1}$, we observe that
\begin{align}\label{eq:lower_iter}
    \left\| \theta_{k+1} - \theta^*(x_{k+1}) \right\|^2 &= \left\| \theta_k - \theta^*(x_{k+1}) - \alpha\nabla_\theta g(x_{k+1},\theta_k) \right\|^2 \nonumber\\
    &= \left\| \theta_k - \theta^*(x_{k+1}) \right\|^2 - 2\alpha \langle \nabla_\theta g(x_{k+1},\theta_k), \theta_k - \theta^*(x_{k+1}) \rangle \nonumber\\
    & \quad + \alpha^2 \left\| \nabla_\theta g(x_{k+1},\theta_k) \right\|^2.
\end{align}
Moreover, from Assumption \ref{assump:g-conditions} we have that
\begin{align}\label{eq:g-sc-lip}
    \langle \nabla_\theta g(x_{k+1},\theta_k), \theta_k - \theta^*(x_{k+1}) \rangle \geq \frac{\mu_g L_g}{\mu_g + L_g} \left\| \theta_k - \theta^*(x_{k+1}) \right\|^2 + \frac{1}{\mu_g + L_g} \left\| \nabla_\theta g(x_{k+1},\theta_k) \right\|^2.
\end{align}
The inequality in \eqref{eq:g-sc-lip} together with \eqref{eq:lower_iter} imply that
\begin{align}\label{eq:lower_convrg}
    \left\| \theta_{k+1} - \theta^*(x_{k+1}) \right\|^2 &\leq \left\| \theta_k - \theta^*(x_{k+1}) \right\|^2 - \frac{2\alpha\mu_gL_g}{\mu_g+L_g} \left\| \theta_k - \theta^*(x_{k+1}) \right\|^2 \nonumber\\
    & \quad + (\alpha^2 - \frac{2\alpha}{\mu_g+L_g}) \left\| \nabla_\theta g(x_{k+1}, \theta_k \right\|^2.
\end{align}
Setting the step-size $\alpha = \frac{2}{\mu_g + L_g}$ in \eqref{eq:lower_convrg} leads to
\begin{align*}
    \left\| \theta_{k+1} - \theta^*(x_{k+1}) \right\|^2 \leq (\frac{\mu_g - L_g}{\mu_g + L_g})^2 \left\| \theta_k - \theta^*(x_{k+1}) \right\|^2.
\end{align*}

Next, recall that $\beta=(L_g-\mu_g)/(L_g+\mu_g)$. Using the triangle inequality and Part (I) of Lemma \ref{lem:theta-gradxL-lip} we conclude that
    \begin{align}\label{eq:one-step-triangle}
    \left\| \theta_{k+1} - \theta^*(x_{k+1}) \right\| &\leq \beta \left\| \theta_k - \theta^*(x_{k+1}) \right\| \nonumber\\
    & \leq \beta \big[\left\| \theta_k - \theta^*(x_k) \right\| + \left\| \theta^*(x_k) - \theta^*(x_{k+1}) \right\|\big] \nonumber\\
    &\leq \beta \big[\left\| \theta_k - \theta^*(x_k) \right\| + \bL_\theta\left\| x_k - x_{k+1} \right\|\big].
\end{align}
Moreover, from the update of $\bx_{k+1}$ in Algorithm \ref{i-BRPD:OPF} and boundedness of $\cX$ we have that $\norm{\bx_{k+1}-\bx_k}\leq \gamma_k D_\cX$. Therefore, using this inequality within \eqref{eq:one-step-triangle} leads to 
\begin{equation*}
\norm{\btheta_{k+1}-\btheta^*(\bx_{k+1})}\leq \beta\norm{\btheta_k-\btheta^*(\bx_k)}+\beta\gamma_k D_\cX \bL_\theta.
\end{equation*}
Therefore, the desired result can be deduced from the above inequality recursively.\\
Moreover, since $\mathcal{X}$ is compact, the sequence $\{x_k\}$ remains in a bounded set, implying that $\|x_i - x_{i+1}\|$ is uniformly bounded for all $i$. Given that $\theta^*(\cdot)$ is a continuous map on the compact set $\mathcal{X}$, it is also bounded. Consequently, since $0 < \beta < 1$, the right-hand side of the inequality remains bounded, ensuring that the sequence $\{\theta_k\}_{k\geq0}$ is a bounded sequence.
\end{proof}

In the next step of analyzing the lower-level approximation, we need to understand how well $w_{k+1}$ tracks $v(x_k,y_k)$. As explained in Section \ref{sec:main-algorithm}, $v(x_k,y_k)$ can be approximated by an iterate $w_{k+1}$ obtained by taking one step of gradient descent. 
\begin{lemma}\label{lem:lower-w-v}
    Let $\{(x_k,w_k,y_k)\}_{k\geq0}$ be the sequence generated by the Algorithm \ref{i-BRPD:OPF} or \ref{i-BRPD:FP} and suppose Assumptions \ref{assump:grad-xytheta-lip}, \ref{assump:g-conditions} hold and $\eta\leq 2/(L_g+\mu_g)$. Then, for any $k\geq 0$,
    \begin{align*}
        \|w_{k+1} - v(x_k,y_k) \| & \leq \rho^{k+1} \|w_0 - v(x_0,y_0)\| + \gamma\bC_{v1} \sum_{i=0}^{k-1} \rho^{k-i} \|s_i - x_i \| +  \bC_{v2} \sum_{i=0}^{k-1} \rho^{k-i} \|y_{i+1} - y_i \|\nonumber\\
        & \quad +\Tilde{C}\left\| \theta_0 - \theta^*(x_0) \right\| k\rho^k + \gamma\Tilde{C}\bL_\theta \sum_{i=0}^k \sum_{j=0}^{i-1} \rho^{k-i} \beta^{i-j} \left\| s_j-x_j\right\|,
    \end{align*}
where $\Tilde{C} = \eta (L_{\theta \theta}^g \frac{C_\theta^\Phi}{\mu_g}+L_{\theta \theta}^\Phi)$ and $\rho \triangleq (1-\eta \mu_g) $. Moreover, the sequence $\{w_{k} \}_{k\geq0}$ is a bounded sequence.
\end{lemma} 
\begin{proof}
    By verifying that $v(x_k,y_k) = v(x_k,y_k) - \eta \big(\nabla_{\theta \theta}^2 g(x_k,\theta^*(x_k))v(x_k,y_k) - \nabla_\theta \Phi(x_k,\theta^*(x_k),y_k)\big)$, and by defining $w_{k+1} = w_k - \eta \big(\nabla_{\theta \theta}^2 g(x_k, \theta_k) w_k - \nabla_\theta \Phi(x_k, \theta_k, y_k)\big)$ we can write
    \begin{align*}
        \|w_{k+1} - v(x_k,y_k) \| &= \| \Big(w_k - \eta \big(\nabla_{\theta \theta}^2 g(x_k, \theta_k) w_k - \nabla_\theta \Phi(x_k, \theta_k, y_k)\big) \Big) \nonumber\\
        & \quad - \Big(v(x_k,y_k) - \eta \big(\nabla_{\theta \theta}^2 g(x_k,\theta^*(x_k))v(x_k,y_k) - \nabla_\theta \Phi(x_k,\theta^*(x_k),y_k)\big) \Big) \| \nonumber\\
        & = \| \Big(I - \eta \nabla_{\theta \theta}^2 g(x_k,\theta_k) \Big)(w_k - v(x_k,y_k)) - \eta \Big(\nabla_{\theta \theta}^2g(x_k,\theta_k) \nonumber\\
        & \quad - \nabla_{\theta \theta}^2g(x_k,\theta^*(x_k))\Big)v(x_k,y_k) + \eta \Big(\nabla_\theta \Phi(x_k,\theta^*(x_k),y_k)-\nabla_\theta\Phi(x_k,\theta_k,y_k)\Big) \|,
    \end{align*}
    where the last equality is obtained by adding and subtracting the term $(I - \eta \nabla_{\theta \theta}^2 g(x_k,\theta_k))v(x_k,y_k)$. Next, using Assumptions \ref{assump:grad-xytheta-lip} and \ref{assump:g-conditions} along with the application of the triangle inequality we obtain
    \begin{align*}
        \|w_{k+1} - v(x_k,y_k) \| &\leq (1-\eta \mu_g) \| w_k - v(x_k,y_k) \| + \eta L_{\theta \theta}^g \|\theta_k  - \theta^*(x_k)\| \|v(x_k,y_k)\| \nonumber\\
        & \quad + \eta L_{\theta \theta}^\Phi \|\theta_k - \theta^*(x_k)\|.
    \end{align*}
    Note that $ \|v(x,y)\| = \| [\nabla_{\theta\theta}^2 g(x,\theta^{*}(x))]^{-1} \nabla_\theta\Phi(x,\theta^*(x),y)\| \leq \frac{C_\theta^\Phi}{\mu_g}$. Now, by adding and subtracting $v(x_{k-1},y_{k-1})$ to the term $\|w_k - v(x_k,y_k) \|$ followed by the triangle inequality application we can conclude that
    \begin{align*}
        \|w_{k+1} - v(x_k,y_k) \| &\leq (1-\eta \mu_g) \|w_k - v(x_{k-1},y_{k-1}) \| + (1-\eta \mu_g) \|v(x_{k-1},y_{k-1}) - v(x_k,y_k)\| \nonumber\\
        & \quad + \eta(L_{\theta \theta}^g \frac{C_\theta^\Phi}{\mu_g}+L_{\theta \theta}^\Phi) \|\theta_k - \theta^*(x_k)\|.
    \end{align*}
    Therefore, using the result of Lemma \ref{lem:lower-v} and \ref{lem:lower_GD_C} we can show
    \begin{align*}
        \|w_{k+1} - v(x_k,y_k) \| &\leq \rho \|w_k - v(x_{k-1},y_{k-1}) \| + \rho (\bC_{v1}\left\| x_i -x_{i-1} \right\| + \bC_{v2}\left\| y_i - y_{i-1} \right\|) \nonumber\\
        & \quad + \eta (L_{\theta \theta}^g \frac{C_\theta^\Phi}{\mu_g}+L_{\theta \theta}^\Phi)(\beta^i \left\| \theta_0 - \theta^*(x_0) \right\| + \bL_\theta \sum_{j=0}^{i-1} \beta^{i-j} \left\| x_j - x_{j+1} \right\|).
    \end{align*}
    Consequently, by continuing the above inequality  recursively from $0$ to $k$, we obtain
    \begin{align*}
        \|w_{k+1} - v(x_k,y_k) \| &\leq \rho^{k+1} \|w_0 - v(x_0,y_0)\| + \sum_{i=0}^k \rho^{k-i+1}(\bC_{v1}\left\| x_i -x_{i-1} \right\| + \bC_{v2}\left\| y_i - y_{i-1} \right\|) \nonumber\\
        & \quad + \sum_{i=0}^k \rho^{k-i} \eta (L_{\theta \theta}^g \frac{C_\theta^\Phi}{\mu_g}+L_{\theta \theta}^\Phi)(\beta^i \left\| \theta_0 - \theta^*(x_0) \right\| + \bL_\theta \sum_{j=0}^{i-1} \beta^{i-j} \left\| x_j - x_{j+1} \right\|) \nonumber\\
        & \leq \rho^{k+1} \|w_0 - v(x_0,y_0)\| + \bC_{v1} \sum_{i=0}^{k-1} \rho^{k-i} \|x_{i+1} - x_i \| +  \bC_{v2} \sum_{i=0}^{k-1} \rho^{k-i} \|y_{i+1} - y_i \|\nonumber\\
        & \quad +\Tilde{C}\left\| \theta_0 - \theta^*(x_0) \right\| \sum_{i=0}^k \rho^{k-i} \beta^i + \Tilde{C}\bL_\theta \sum_{i=0}^k \rho^{k-i} \sum_{j=0}^{i-1}\beta^{i-j} \left\| x_j - x_{j+1} \right\| \nonumber\\
        & \leq \rho^{k+1} \|w_0 - v(x_0,y_0)\| + \bC_{v1} \sum_{i=0}^{k-1} \rho^{k-i} \|x_{i+1} - x_i \| +  \bC_{v2} \sum_{i=0}^{k-1} \rho^{k-i} \|y_{i+1} - y_i \|\nonumber\\
        & \quad +\Tilde{C}\left\| \theta_0 - \theta^*(x_0) \right\| k\rho^k + \Tilde{C}\bL_\theta \sum_{i=0}^k \sum_{j=0}^{i-1} \rho^{k-i} \beta^{i-j} \left\| x_j - x_{j+1} \right\|,
    \end{align*}
    where $\Tilde{C} = \eta (L_{\theta \theta}^g \frac{C_\theta^\Phi}{\mu_g}+L_{\theta \theta}^\Phi)$ and in the last inequality we used the fact that $\beta\leq \rho$. 
    The result follows by rearranging the terms on the RHS and $x_{k+1}-x_k=\gamma(s_k-x_k)$ for any $k\geq 0$.\\ 
    Moreover, since $\mathcal{X}$ and $\mathcal{Y}$ are compact, and $v(\cdot,\cdot)$ is a continuous map on these sets, it follows that $v$ is bounded. Additionally, all terms on the right-hand side are bounded implying that the sequence $\{w_k\}_{k\geq0}$ is also bounded.
\end{proof}

\subsection{Analysis of Tracking the Optimal Solution Trajectory of the Maximization Component}\label{subsec:max-component}
In this section, we examine the properties of the regularized optimal solution in the maximization component of \eqref{eq:minmax-lmu} as discussed in Section \ref{sec:main-algorithm}. In particular, we show that the regularized optimal solution $y^*_{\mu}$ is Lipschitz continuous, and the function $f_{\mu}$ has a
Lipschitz continuous gradient, in Lemmas  \ref{lem: ymu-lip} and \ref{lem:grad-fmu-lip}, respectively. Although similar results have been shown in the literature \cite{sinha2018certifying,zhao2023primal}, here, we provide the proof for completeness. 
\begin{lemma}\label{lem: ymu-lip}
    The solution map $y^*_\mu : \mathcal{X} \rightarrow \mathcal{Y}$ is Lipschitz continuous. Suppose Assumptions \ref{assump:grad-xytheta-lip}, \ref{assump:g-conditions} hold and $\theta^*(x)$ is Lipschitz continuous map with constant $\bL_\theta$. In particular, for any $x,\overline{x} \in \mathcal{X}$
    \begin{align*}
        \left\| y^*_\mu(\overline{x})-y^*_\mu(x) \right\| \leq \frac{L_{y\mu}}{\mu}\left\| x - \overline{x} \right\|,
    \end{align*}
    where $L_{y\mu} = L_{yx}^\Phi + \bL_\theta L_{y\theta}^\Phi$.
\end{lemma}
\begin{proof}
    Since $\mathcal{L}_\mu (x,\cdot)$ is strongly concave for any $x\in \mathcal{X}$, we have that
    \begin{align} \label{eq: Lmu-strng-conc}
    (y^*_\mu(\overline{x})-y^*_\mu(x))^\top\big[\nabla_y\mathcal{L}_\mu(x,y^*_\mu(\overline{x})) - \nabla_y\mathcal{L}_\mu(x,y^*_\mu(x))\big] + \mu \left\| y^*_\mu(\overline{x})-y^*_\mu(x) \right\|^2 \leq 0.
\end{align}
Moreover, the optimality of $y^*_\mu (\overline{x})$ and $y^*_\mu (x)$ given that $\mathcal{L}_\mu (x,y)= \mathcal{L}(x,y) - \frac{\mu}{2} \left\| y - y_0 \right\|^2$ implies that for any $y\in \mathcal{Y}$,
\begin{align}\label{eq: opt-ymu1}
    (y-y^*_\mu(\overline{x}))^\top \nabla_y \mathcal{L}_\mu(\overline{x},y^*_\mu(\overline{x}))\leq0,
\end{align}
\begin{align}\label{eq: opt-ymu2}
    (y-y^*_\mu(x))^\top \nabla_y \mathcal{L}_\mu(x,y^*_\mu(x))\leq0.
\end{align}
Let $y=y^*_\mu(x)$ in \eqref{eq: opt-ymu1} and $y=y^*_\mu(\overline{x})$ in \eqref{eq: opt-ymu2} and summing up two inequalities, we obtain
\begin{align} \label{eq: ym-sum}
    (y^*_\mu(x)-y^*_\mu(\overline{x}))^\top\Big(\nabla_y\mathcal{L}_\mu(\overline{x},y^*_\mu(\overline{x})) - \nabla_y\mathcal{L}_\mu(x,y^*_\mu(x))\Big)\leq0.
\end{align}
By combining \eqref{eq: ym-sum} and \eqref{eq: Lmu-strng-conc} we have
\begin{align*}
    \mu \left\| y^*_\mu(\overline{x})-y^*_\mu(x) \right\|^2 \leq (y^*_\mu(\overline{x})-y^*_\mu(x))^\top\Big(\nabla_y\mathcal{L}_\mu(\overline{x},y^*_\mu(\overline{x})) - \nabla_y\mathcal{L}_\mu(x,y^*_\mu(\overline{x}))\Big).
\end{align*}
We know from \eqref{eq:grady-L} that $\nabla_y\mathcal L(x,y) = \nabla_y \Phi(x,\theta^*(x),y)$ and by defining $\Phi_\mu \triangleq \Phi - \frac{\mu}{2}\left\| y - y_0 \right\|^2$ we have
\begin{align*} 
    \mu \left\| y^*_\mu(\overline{x})-y^*_\mu(x) \right\|^2 \leq (y^*_\mu(\overline{x})-y^*_\mu(x))^\top\Big(\nabla_y\Phi_\mu(\overline{x},\theta^*(\overline{x}),y^*_\mu(\overline{x})) - \nabla_y\Phi_\mu(x,\theta^*(x),y^*_\mu(\overline{x}))\Big).
\end{align*}
Using Assumption \ref{assump:grad-xytheta-lip}-(3) we can get
\begin{align*} 
    \mu \left\| y^*_\mu(\overline{x})-y^*_\mu(x) \right\|^2 \leq \left\| y^*_\mu(\overline{x})-y^*_\mu(x) \right\|(L_{yx}^\Phi \left\| x - \overline{x} \right\| + L_{y\theta}^\Phi \left\| \theta^*(x) - \theta^*(\overline{x})\right\|),
\end{align*}
\begin{align*}
    \left\| y^*_\mu(\overline{x})-y^*_\mu(x) \right\| \leq \frac{L_{yx}^\Phi}{\mu}\left\| x - \overline{x} \right\| + \frac{L_{y\theta}^\Phi}{\mu}\left\| \theta^*(x) - \theta^*(\overline{x})\right\|.
\end{align*}
Considering Lemma \ref{lem:theta-gradxL-lip}-(I) we can show that
\begin{align*}
    \left\| y^*_\mu(\overline{x})-y^*_\mu(x) \right\| \leq (\frac{L_{yx}^\Phi+\bL_\theta L_{y\theta}^\Phi}{\mu})\left\| x - \overline{x} \right\|.
\end{align*}
\end{proof}
\begin{lemma}\label{lem:grad-fmu-lip}
    Suppose Assumptions \ref{assump:grad-xytheta-lip}, \ref{assump:g-conditions} hold. The function $f_\mu(\cdot)$ is differentiable on an open set containing $\mathcal{X}$ and $\nabla f_\mu(x) = \nabla_x\mathcal{L}(x,y^*_\mu(x))$ where $y^*_\mu(x) \triangleq \argmin_{y\in\mathcal{Y}}\mathcal{L}_\mu(x,y)$. Moreover, $f_\mu$ has a Lipschitz continuous gradient with constant $\bL_{\mathcal L_1} + \frac{\bL_{\mathcal L_2} L_{y\mu}}{\mu}$.
\end{lemma}
\begin{proof}
    From Danskins's theorem \cite{bernhard1995theorem} one can obtain $f_\mu(\cdot)$ is differentiable and $\nabla f_\mu(x) = \nabla_x\mathcal{L}(x,y^*_\mu(x))$. Therefore, we have for any $x,x'\in\cX$, 
    \begin{align*}
        \left\| \nabla f_\mu(x) - \nabla f_\mu(x') \right\| = \left\| \nabla_x \mathcal{L}(x,y^*_\mu(x)) - \nabla_x\mathcal{L}(x',y^*_\mu(x')) \right\|.
    \end{align*}
    Using Lemma \ref{lem:theta-gradxL-lip}-(II) we can show
    \begin{align*}
        \left\| \nabla f_\mu(x) - \nabla f_\mu(x') \right\| \leq \bL_{\mathcal L_1} \left\| x - x' \right\| + \bL_{\mathcal L_2}\left\| y^*_\mu(x) - y^*_\mu(x') \right\|.
    \end{align*}
    Considering the result in Lemma \ref{lem: ymu-lip} we have
    \begin{align*}
        \left\| \nabla f_\mu(x) - \nabla f_\mu(x') \right\| \leq \bL_{\mathcal L_1} \left\| x - x' \right\| + \bL_{\mathcal L_2}\frac{L_{y\mu}}{\mu}\left\| x - x' \right\|,
    \end{align*}
    \begin{align*}
        \left\| \nabla f_\mu(x) - \nabla f_\mu(x') \right\| \leq (\bL_{\mathcal L_1} + \frac{\bL_{\mathcal L_2} L_{y\mu}}{\mu})\left\| x - x' \right\|.
    \end{align*}
\end{proof}
Next, in the subsequent lemmas, we establish the bound on the distance between the iterative solution $y_k$ and the optimal solution $y^*_\mu(x_k)$. Moreover, by deriving a bound on the consecutive iterates of $y_k$, we demonstrate that the error between them can be controlled using parameter $\gamma$.   
\begin{lemma} \label{lem: lower-y-ymu}
    Suppose Assumptions \ref{assump:grad-xytheta-lip}, \ref{assump:g-conditions} hold. Let $\{(x_k,y_k)\}_{k\geq0}$ be the sequence generated by the Algorithm \ref{i-BRPD:OPF} or \ref{i-BRPD:FP} with $\sigma_k = \sigma\leq \frac{2}{L_{yy}^\Phi+2\mu}$. Then for any $k\geq0$,
    \begin{align*}
    \left\| y_k - y^*_\mu(x_k) \right\|&\leq \rho_D\|y_{k-1}-y^*_\mu(x_k)\|+\sigma L^\Phi_{y\theta}\Big(\beta^k \left\| \theta_0 - \theta^*(x_0) \right\| + \gamma \bL_\theta \sum_{i=0}^{k-1} \beta^{k-i} \left\| s_i - x_i \right\| \Big),
\end{align*}
where $\rho_D\triangleq L_{yy}^\Phi/(L_{yy}^\Phi+2\mu)$. Moreover, when $L_{yy}^\Phi=0$, selecting $\sigma_k=\sigma=\frac{1}{\mu}$ and $y_0=\mathbf{0}$ we have that for any $k\geq 0$, 
\begin{align*}
    \norm{y_k-y^*_\mu(x_k)}\leq \frac{L^\Phi_{y\theta}}{\mu}\Big(\beta^k \left\| \theta_0 - \theta^*(x_0) \right\| + \gamma \bL_\theta \sum_{i=0}^{k-1} \beta^{k-i} \left\| s_i - x_i \right\| \Big).
\end{align*} 
\end{lemma}
\begin{proof}
    From the optimality condition we have that $y^*_\mu(x) = \mathcal{P}_{\mathcal{Y}}\Big(y^*_\mu(x) + \sigma \nabla_y \mathcal{L}_\mu(x,y^*_\mu(x))\Big)$ for any $x\in \mathcal{X}$. Moreover, from the update of $y_k$ as $y_k = \mathcal P_\mathcal{Y}(y_{k-1} + \sigma G^y_{k-1})$ where $G^y_k\triangleq \nabla_y\Phi(x_k,\theta_k,y_k)-\mu(y_k - y_0)$ for any $k\geq 0$, and the non-expansivity of the projection operator, by adding and subtracting $\nabla_y \mathcal{L}_\mu(x_k,y_{k-1})$, we have
    \begin{align*}
    \left\| y_k - y^*_\mu(x_k) \right\| &\leq \left\| y_{k-1} + \sigma \Big(G^y_{k-1} + \nabla_y \mathcal{L}_\mu(x_k,y_{k-1}) - \nabla_y \mathcal{L}_\mu(x_k,y_{k-1})\Big) \right. \nonumber\\
    & \quad - \left. \Big(y^*_\mu(x_k) + \sigma \nabla_y \mathcal{L}_\mu(x, y^*_\mu(x_k))\Big) \right\| \nonumber\\
    &\leq \|\nabla g_k(y_{k-1})-\nabla g_k(y^*_\mu(x_k))\|+\sigma\|G_{k-1}^y-\nabla_y\mathcal L_\mu(x_k,y_{k-1})\| \nonumber\\
    &\leq \rho_D\|y_{k-1}-y^*_\mu(x_k)\|+\sigma\|G_{k-1}^y-\nabla_y\mathcal L_\mu(x_k,y_{k-1})\|,
\end{align*}
where $g_k(y)\triangleq \frac{1}{2}\|y\|^2+\sigma \mathcal L_\mu(x_k,y)$ and $\rho_D\triangleq \max\{| 1 - \sigma(L^\Phi_{yy} + \mu) |, | 1 - \sigma\mu|\}=L_{yy}^\Phi/(L_{yy}^\Phi+2\mu)$. Next, using the definition of $G^y_k$ and $\nabla_y\mathcal L$ and considering Assumption \ref{assump:grad-xytheta-lip}-(3) we obtain
\begin{align*}
    \left\| y_k - y^*_\mu(x_k) \right\|&\leq\rho_D\|y_{k-1}-y^*_\mu(x_k)\|+\sigma\|\nabla_y\Phi(x_k,\theta_k,y_{k-1})-\nabla_y\Phi(x_k,\theta^*(x_k),y_{k-1})\| \nonumber\\
    &\leq \rho_D\|y_{k-1}-y^*_\mu(x_k)\|+\sigma L^\Phi_{y\theta}\|\theta_k-\theta^*(x_k)\|.
\end{align*}
Using Lemma \ref{lem:lower_GD_C} we can show
\begin{align*}
    \left\| y_k - y^*_\mu(x_k) \right\|&\leq \rho_D\|y_{k-1}-y^*_\mu(x_k)\|+\sigma L^\Phi_{y\theta}(\beta^k \left\| \theta_0 - \theta^*(x_0) \right\| + \bL_\theta \sum_{i=0}^{k-1} \beta^{k-i} \left\| x_i - x_{i+1} \right\|).
\end{align*}
Considering $x_{k+1} - x_k = \gamma (s_k-x_k)$ 
we can obtain the desired result.
\end{proof}
\begin{lemma} \label{lem: y-ymu-bound}
    Under the premises of Lemma \ref{lem: lower-y-ymu}, for any $k\geq 0$ we have
    \begin{align}\label{eq:final-bound-yk}
    \left\| y_k - y^*_\mu(x_k) \right\| &\leq \rho_D^k \left\| y_0 - y^*_\mu(x_0) \right\| + \frac{\gamma L_{y\mu}}{\mu} \sum_{i=0}^{k-1} \rho_D^{k-i} \left\| s_i - x_i \right\| + \sigma L_{y\theta}^\Phi \left\|\theta_0 - \theta^*(x_0)\right\|\sum_{i=1}^{k}\rho^{k-i}_D\beta^i \nonumber\\
    & \quad + \sigma\gamma L_{y\theta}^\Phi \bL_\theta\sum_{i=1}^k\rho^{k-i}_D  \sum_{j=0}^{i-1} \beta^{i-j} \left\| s_j - x_j \right\| .
\end{align}
Moreover, for any $K\geq 1$,
\begin{equation}\label{eq:sum-Ak-y*}
    \sum_{k=0}^{K-1}\left\| y_k - y^*_\mu(x_k) \right\|\|s_k-x_k\|\leq \Gamma_1+\Gamma_2\sum_{k=0}^{K-1} \|s_k-x_k\|^2,
\end{equation}
where $ \Gamma_1\triangleq \frac{1}{1-\rho_D}\left\| y_0 - y^*_\mu(x_0) \right\|D_\cX+\frac{\sigma L_{y\theta}^\Phi \beta D_\cX}{(1-\rho_D)(1-\beta)} \left\|\theta_0 - \theta^*(x_0)\right\|=\cO(\frac{\kappa_g}{\mu}L_{y\theta}^\Phi)$ and $\Gamma_2\triangleq \frac{\gamma L_{y\mu}}{\mu (1-\rho_D)}+\frac{\sigma \gamma \beta L_{y\theta}^\Phi \bL_\theta}{(1-\beta)(1-\rho_D)}=\cO\left(\frac{\kappa_gL_{y\theta}^\Phi+L_{y x}^\Phi}{\mu^2}\gamma +\frac{\kappa_g^2L_{y\theta}^\Phi}{\mu}\gamma\right)$.
Furthermore, if $L^\Phi_{yy} = 0$, in \eqref{eq:sum-Ak-y*} we have that 
$\Gamma_1 = \frac{\sigma L^\Phi_{y\theta}}{(1-\beta)}\|\theta_0 - \theta^*(x_0)\|D_{\mathcal{X}} = \mathcal{O}(\frac{\kappa_g L^\Phi_{y\theta}}{\mu})$ and $\Gamma_2 = \frac{\gamma L_{y\mu}}{\mu}+ \frac{\sigma \gamma\beta L^\Phi_{y\theta}  \bL_\theta}{(1-\beta)} = \cO(\frac{\kappa_g^2 L^\Phi_{y \theta}}{\mu}\gamma)$.
\end{lemma}
\begin{proof}
    From Lemma \ref{lem: lower-y-ymu} we have 
    \begin{align*}
        \left\| y_k - y^*_\mu(x_k) \right\| \leq \rho_{D} \left\| y_{k-1} - y^*_\mu(x_k) \right\| + \sigma L_{y\theta}^\Phi(\beta^k \left\|\theta_0 - \theta^*(x_0)\right\| + \gamma\bL_\theta \sum_{i=0}^{k-1} \beta^{k-i} \left\| s_i - x_i \right\|).
    \end{align*}
    By adding and subtracting $y^*_\mu(x_{k-1})$ and using triangle inequality we obtain
    \begin{align*}
        \left\| y_k - y^*_\mu(x_k) \right\| &\leq \rho_D (\left\| y_{k-1} - y^*_\mu(x_{k-1}) \right\| + \left\| y^*_\mu(x_k) - y^*_\mu(x_{k-1}) \right\|) + \sigma L_{y\theta}^\Phi\Big(\beta^k \left\|\theta_0 - \theta^*(x_0)\right\| \nonumber\\
        & \quad + \gamma\bL_\theta \sum_{i=0}^{k-1} \beta^{k-i} \left\| s_i - x_i \right\|\Big) \nonumber\\
        &\leq \rho_D \left\| y_{k-1} - y^*_\mu(x_{k-1}) \right\| + \frac{L_{y\mu}}{\mu}\rho_D \left\| x_k - x_{k-1} \right\| + \sigma L_{y\theta}^\Phi\Big(\beta^k \left\|\theta_0 - \theta^*(x_0)\right\|\nonumber\\
        &\quad +\gamma\bL_\theta \sum_{i=0}^{k-1} \beta^{k-i} \left\| s_i - x_i \right\| \Big),
    \end{align*}
    where in the last inequality, we applied Lemma \ref{lem: ymu-lip}. 
    
    Continuing the above inequalities recursively we obtain 
    \begin{align*}
    \left\| y_k - y^*_\mu(x_k) \right\| &\leq \rho_D^k \left\| y_0 - y^*_\mu(x_0) \right\| + \frac{\gamma L_{y\mu}}{\mu} \sum_{i=0}^{k-1} \rho_D^{k-i} \left\| s_i - x_i \right\| + \sum_{i=1}^{k} \sigma L_{y\theta}^\Phi \Big(\rho^{k-i}_D\beta^i \left\|\theta_0 - \theta^*(x_0)\right\| \nonumber\\
    & \quad + \rho^{k-i}_D \gamma\bL_\theta \sum_{j=0}^{i-1} \beta^{i-j} \left\| s_j - x_j \right\| \Big),
\end{align*}%
which leads to the result in \eqref{eq:final-bound-yk} by rearranging the terms.




Moreover, using Lemma \ref{lem:sum-sum-linear-comb} and the fact that $\|s_k - x_k\| \leq D_{\mathcal{X}}$, one can verify that for any $K\geq 1$, 
\begin{align*}
    \sum_{k = 0}^{K-1}\left\| y_k - y^*_\mu(x_k) \right\| \|s_k-x_k\|&\leq  \frac{1}{1-\rho_D}\left\| y_0 - y^*_\mu(x_0) \right\|D_\cX + \frac{\gamma L_{y\mu}}{\mu (1-\rho_D)}\sum_{k=0}^{K-1} \|s_k-x_k\|^2  \nonumber \\
    & \quad + \frac{\sigma L_{y\theta}^\Phi \beta D_\cX}{(1-\rho_D)(1-\beta)} \left\|\theta_0 - \theta^*(x_0)\right\|  + \frac{\sigma \gamma \beta L_{y\theta}^\Phi \bL_\theta}{(1-\beta)(1-\rho_D)} \sum_{k=0}^{K-1} \|s_k-x_k\|^2 .
    \nonumber \\
\end{align*}
Finally, if $L^\Phi_{yy} = 0$, then $\rho_D = 0$, selecting $\sigma_k = \sigma = \frac{1}{\mu}$, we have
\begin{align*}
    \sum_{k = 0}^{K-1}\left\| y_k - y^*_\mu(x_k) \right\| \|s_k-x_k\|&\leq \frac{L^\Phi_{y\theta}}{\mu (1-\beta)}\|\theta_0 - \theta^*(x_0)\|D_{\mathcal{X}} + \frac{L^\Phi_{y\theta} \gamma \bL_\theta \beta}{\mu (1-\beta)} \sum_{k=0}^{K-1} \|s_k-x_k\|^2 .
\end{align*}
Note that $\bL_\theta=\cO(\kappa_g)$, $L_{y\mu}=\cO(\kappa_g)$, $1-\rho_D=\cO(\mu)$, and $1-\beta=\cO(1/\kappa_g)$. 
\end{proof}

\begin{lemma} \label{lem: y-bound}
    Suppose Assumptions \ref{assump:grad-xytheta-lip}, \ref{assump:g-conditions} hold. For any $k\geq0$ we have
    \begin{align} \label{eq: lower-y}
        \left\| y_{k+1} - y_k \right\| &\leq 2\rho_D^k \left\| y_0 - y^*_\mu(x_0) \right\| + \frac{2\gamma L_{y\mu}}{\mu} \sum_{i=0}^{k-1} \rho_D^{k-i} \left\| s_i - x_i \right\| + 2\sigma L_{y\theta}^\Phi \left\|\theta_0 - \theta^*(x_0)\right\|\sum_{i=1}^{k}\rho^{k-i}_D\beta^i \nonumber\\
        & \quad + 2\sigma\gamma L_{y\theta}^\Phi \bL_\theta\sum_{i=1}^k\rho^{k-i}_D  \sum_{j=0}^{i-1} \beta^{i-j} \left\| s_j - x_j \right\|+ \frac{L_{y\mu}}{\mu}\gamma \norm{s_k-x_k}.
    \end{align}
    Moreover, for any $K\geq 1$
    \begin{align}\label{eq:sum-Ak-yk}
        \sum_{k=0}^{K-1}\norm{s_k-x_k}\sum_{i=0}^{k-1}\rho^{k-i}\norm{y_{i+1}-y_i}\leq 2\Gamma_1+\Gamma_3\sum_{k=0}^{K-1} \|s_k-x_k\|^2,
    \end{align} 
    where $\Gamma_3\triangleq \frac{2\Gamma_2}{1-\rho}+\frac{L_{y\mu}\gamma}{(1-\rho)\mu}$, and $\Gamma_1,\Gamma_2$ are defined in Lemma \ref{lem: y-ymu-bound} and $\Gamma_3= \cO(\frac{\kappa_g^2L_{y\theta}^\Phi+\kappa_g L_{yx}^\Phi}{\mu^2}\gamma+\frac{\kappa_g^3L_{y\theta}^\Phi}{\mu}\gamma)$. 
    Furthermore, if $L^\Phi_{yy} = 0$, then $\rho_D = 0$, and the bound attained in \eqref{eq: lower-y} can be rewritten as
    \begin{align*} 
        \left\| y_{k+1} - y_k \right\| &\leq \frac{2L_{y\theta}^\Phi}{\mu}  \left\|\theta_0 - \theta^*(x_0)\right\|\beta^k + \frac{2\gamma L_{y\theta}^\Phi \bL_\theta}{\mu} \sum_{i=0}^{k-1} \beta^{k-i} \left\| s_i - x_i \right\| + \frac{L_{y\mu}}{\mu}\gamma \norm{s_k-x_k}.
    \end{align*}
    In this case, $\Gamma_3 = \cO(\frac{\kappa_g^3 L^\Phi_{y \theta}}{\mu}\gamma + \frac{\kappa_g^2 L^\Phi_{y \theta} + \kappa_g L^\Phi_{y x}}{\mu}\gamma)$.
\end{lemma}
\begin{proof}
    Using the triangle inequality, one can observe that for any $k\geq0$,
    \begin{align*}
        \left\| y_{k+1} - y_k \right\| \leq \left\| y_{k+1} - y^*_\mu (x_{k+1}) \right\| + \left\| y^*_\mu (x_{k}) - y_k \right\| + \left\| y^*_\mu (x_{k+1}) - y^*_\mu (x_{k}) \right\|.
    \end{align*}
    Using Lemma \ref{lem: ymu-lip} and \ref{lem: y-ymu-bound}, we can obtain
    \begin{align*}
        \left\| y_{k+1} - y_k \right\| &\leq 2\rho_D^k \left\| y_0 - y^*_\mu(x_0) \right\| + \frac{2\gamma L_{y\mu}}{\mu} \sum_{i=0}^{k-1} \rho_D^{k-i} \left\| s_i - x_i \right\| + 2\sigma L_{y\theta}^\Phi \left\|\theta_0 - \theta^*(x_0)\right\|\sum_{i=1}^{k}\rho^{k-i}_D\beta^i \nonumber\\
        & \quad + 2\sigma\gamma L_{y\theta}^\Phi \bL_\theta\sum_{i=1}^k\rho^{k-i}_D  \sum_{j=0}^{i-1} \beta^{i-j} \left\| s_j - x_j \right\|+ \frac{L_{y\mu}}{\mu}\gamma \norm{s_k-x_k}.
    \end{align*}

    Moreover, the second part of the lemma can be proved with a similar discussion as presented in the proof of Lemma \ref{lem: y-ymu-bound}, by employing the results in Lemma \ref{lem:sum-sum-linear-comb}. 
\end{proof}

\subsection{Analysis of Gradients Estimation}\label{subsec:gradients-estimation}
In this section, we aim to assess how accurately the gradient estimates, denoted as $G^x_k$ and $G^y_k$, approximate the implicit gradients $\nabla_x\mathcal L(x_k,y_k)$ and $\nabla_y\mathcal L(x_k,y_k)$, respectively. Through Lemmas \ref{lem: err-gradxL} and \ref{lem: err-gradyL}, we analyze the bounds on the error in estimating these gradients. We demonstrate that the errors may depend linearly on certain terms or involve consecutive terms, helping us in controlling the error. Later, in the convergence analysis, we will also specify the step-sizes to ensure convergence.  
\begin{lemma}\label{lem: err-gradxL}
    Suppose Assumptions \ref{assump:grad-xytheta-lip} and \ref{assump:g-conditions} hold. Moreover, let $\{x_k,\theta_k,w_k\}_{k\geq0}$ be the sequence generated by Algorithm \ref{i-BRPD:OPF} or \ref{i-BRPD:FP}. Then, for any $k\geq0$
    \begin{align}\label{eq: grad-xL-err}
         \left\| \nabla_x\mathcal L(x_k,y_k) - G^x_k \right\| 
    & \leq \Tilde{C}_2\beta^k \left\| \theta_0 - \theta^*(x_0) \right\| + \Tilde{C}_2\gamma \bL_\theta \sum_{i=0}^{k-1} \beta^{k-i} \left\| s_i - x_i \right\| \nonumber\\
    & \quad + C_{\theta x}^g \Big(\rho^{k+1} \|w_0 - v(x_0,y_0)\| + \gamma \bC_{v1} \sum_{i=0}^{k-1} \rho^{k-i} \|s_i - x_i \| \nonumber\\
    &\quad +  \bC_{v2} \sum_{i=0}^{k-1} \rho^{k-i} \|y_{i+1} - y_i \| +\Tilde{C}\left\| \theta_0 - \theta^*(x_0) \right\| k\rho^k \nonumber\\
    & \quad + \gamma \Tilde{C}\bL_\theta \sum_{i=0}^k\sum_{j=0}^{i-1} \rho^{k-i} \beta^{i-j} \left\| s_j -x_j \right\| \Big),
    \end{align}
    where $\Tilde{C}_2\triangleq \frac{C_\theta^\Phi}{\mu_g} L_{\theta x}^g + L_{x\theta}^\Phi$. Moreover, for any $K\geq 1$ 
    \begin{align}\label{eq:sum-Ak-grad}
        \sum_{k=0}^{K-1}\norm{s_k-x_k}\norm{ \nabla_x\mathcal L(x_k,y_k) - G^x_k}\leq \Gamma_4+2\bC_{v2}\Gamma_1+(\bC_{v2}\Gamma_3+\Gamma_5)\sum_{k=0}^{K-1}\norm{s_k-x_k}^2,
    \end{align} 
    where $\Gamma_1,\Gamma_3$ are defined in Lemma \ref{lem: y-ymu-bound} and \ref{lem: y-bound}, respectively, $\Gamma_4\triangleq D_\cX\norm{\theta_0-\theta^*(x_0)}(\frac{\Tilde{C}_2}{1-\beta}+\frac{\rho\Tilde{C}}{(1-\rho)^2})+ \frac{\rho C_{\theta x}^g D_\cX}{1-\rho}\norm{w_0-v(x_0,y_0)}$ and $\Gamma_5\triangleq \frac{\gamma\beta \bL_\theta}{1-\beta}(1+\frac{\Tilde{C}}{1-\rho})+\frac{\rho\gamma \bC_{v1}C_{\theta x}^g}{1-\rho}$. In particular, $\Gamma_4+2\bC_{v2}\Gamma_1=\cO(\kappa_g^3+L_{y\theta}^\Phi \kappa_g^2/\mu)$ and $\bC_{v2}\Gamma_3+\Gamma_5=\cO\left(\gamma(\frac{\kappa_g^3}{\mu^2}+\frac{\kappa_g^4}{\mu})(L_{y\theta}^\Phi)^2+\gamma\frac{\kappa_g^2}{\mu^2}L_{y\theta}^\Phi L_{yx}^\Phi+\gamma\kappa_g^4\right)$. Moreover, in case of $L^\Phi_{yy} = 0$, we have  $\bC_{v2}\Gamma_3+\Gamma_5=\cO\left(\gamma\frac{ \kappa_g^4}{\mu}(L_{y\theta}^\Phi)^2+\gamma\frac{\kappa_g^2}{\mu}L_{y\theta}^\Phi L_{yx}^\Phi+\gamma\kappa_g^4\right)$.
\end{lemma}
\begin{proof}
    We begin the proof by considering the definition of $\nabla_x\mathcal L(x_k,y_k)$ and $G^x_k$ followed by a triangle inequality to obtain
    \begin{align*}
    \left\| \nabla_x\mathcal L(x_k,y_k) - G^x_k \right\| &= \left\| \nabla_x \Phi(x_k,\theta^*(x_k),y_k) -\nabla_{\theta x}^2 g(x_k,\theta^{*}(x_k)) v(x_k,y_k) \right. \nonumber\\
    & \quad - \left. [\nabla_x \Phi(x_k,\theta_k,y_k) -\nabla_{\theta x}^2 g(x_k,\theta_k)w_{k+1}] \right\| \nonumber\\
    & \leq \left\| \nabla_x \Phi(x_k,\theta^*(x_k),y_k) - \nabla_x \Phi(x_k,\theta_k,y_k) \right\| + \left\| \nabla_{\theta x}^2 g(x_k,\theta_k)w_{k+1} \right. \nonumber\\
    & \quad - \left. \nabla_{\theta x}^2 g(x_k,\theta^{*}(x_k)) v(x_k,y_k) \right\|.
\end{align*}
Combining Assumption \ref{assump:grad-xytheta-lip}-(1) together with adding and subtracting $\nabla_{\theta x}^2 g(x_k,\theta_k)v(x_k,y_k)$ to the second term of RHS lead to
\begin{align*} 
    \left\| \nabla_x\mathcal L(x_k,y_k) - G^x_k \right\| &\leq L_{x\theta}^\Phi \left\| \theta_k - \theta^*(x_k)\right\| + \left\| \nabla_{\theta x}^2 g(x_k,\theta_k) [w_{k+1}-v(x_k,y_k)] \right. \nonumber\\
    & \quad + \left. [\nabla_{\theta x}^2 g(x_k,\theta_k) - \nabla_{\theta x}^2 g(x_k,\theta^{*}(x_k))]v(x_k,y_k) \right\| \nonumber\\
    & \leq L_{x\theta}^\Phi \left\| \theta_k - \theta^*(x_k)\right\| + C_{\theta x}^g \left\| w_{k+1}-v(x_k,y_k) \right\| + \frac{C_\theta^\Phi}{\mu_g} L_{\theta x}^g \left\|\ \theta_k - \theta^*(x_k)\right\|,
\end{align*}
where the last inequality is obtained using Assumption \ref{assump:g-conditions} and the triangle inequality. Next, utilizing Lemma \ref{lem:lower_GD_C} and \ref{lem:lower-w-v}, we can further provide upper-bounds for the terms in RHS of the above inequality to obtain the desired result in \eqref{eq: grad-xL-err}.

To show the result in \eqref{eq:sum-Ak-grad}, note that $\sum_{k=0}^{K-1}k\rho^k\leq \rho/(1-\rho)^2$. Moreover, one can multiply $\norm{s_k-x_k}$ to each term on the RHS of \eqref{eq: grad-xL-err} and sum over $k$ from $0$ to $K-1$ and use Lemma \ref{lem:sum-sum-linear-comb} and \eqref{eq:sum-Ak-yk}. Finally, note that $\bC_{v1}=\cO(\kappa_g^3)$, $\bC_{v2}=\cO(\kappa_g)$, $\Tilde{C}=\cO(\kappa_g)$, $\Tilde{C}_2=\cO(\kappa_g)$, and $1-\rho=\cO(1/\kappa_g)$. 
\end{proof}
\begin{lemma} \label{lem: err-gradyL}
    Using the Assumption \ref{assump:grad-xytheta-lip}-(3), we can find a bound of $\left\| \nabla_y\mathcal L(x_k,y_k) - G_k^y \right\|$ as follows
    \begin{align*}
        \left\| \nabla_y\mathcal L(x_k,y_k) - G_k^y \right\| \leq \beta^k \left\| \theta_0 - \theta^*(x_0) \right\| + \gamma L_{y\theta}^\Phi\bL_\theta \sum_{i=0}^{k-1} \beta^{k-i} \left\| s_i - x_i \right\| + \mu D_\mathcal{Y}.
    \end{align*}
\end{lemma}
\begin{proof}
Following the definition of $\nabla_y\mathcal L(x_k,y_k)$ in \eqref{eq:grady-L} and $G_k^y$ in line 6 of the algorithms, we have that
    \begin{align*}
        \left\| \nabla_y\mathcal L(x_k,y_k) - G_k^y \right\| &= \left\| \nabla_y \Phi(x_k,\theta^*(x_k),y_k) - \nabla_y \Phi(x_k,\theta_k,y_k) + \mu(y_k-y_0)\right\| \nonumber\\
        & \leq L_{y\theta}^\Phi \left\| \theta_k - \theta^*(x_k)\right\|+\mu D_\mathcal{Y},
    \end{align*}
    where in the last inequality, we use Lipschitz continuity of $\grad_y\Phi(x,\cdot,y)$ and boundedness of set $\cY$.
    Next, utilizing Lemma \ref{lem:lower_GD_C}, we can further provide upper-bounds for the term in RHS of the above inequality as follows
    \begin{align*}
        \left\| \nabla_y\mathcal L(x_k,y_k) - G_k^y \right\| \leq L_{y\theta}^\Phi (\beta^k \left\| \theta_0 - \theta^*(x_0) \right\| + \bL_\theta \sum_{i=0}^{k-1} \beta^{k-i} \left\| x_i - x_{i+1} \right\|) + \mu D_\mathcal{Y}.
    \end{align*}
    The result follows by noting that $\norm{x_{k+1}-x_k}= \gamma \norm{s_k-x_k}$ for any $k\geq 0$.
\end{proof}
 
 \section{Convergence Results}\label{sec:convg-results}
In this section, we commence by analyzing the complexity of the i-BRPD:OPF method in Subsection \ref{subsec:conv1}. Later, in Subsection \ref{subsec:conv2}, we delve into the iteration complexity analysis of the i-BRPD:FP method across various setups.
\subsection{Convergence Rate of Algorithm \ref{i-BRPD:OPF}}\label{subsec:conv1}
We provide upper bounds for the primal and dual gap functions for the iterates generated by Algorithm \ref{i-BRPD:OPF} in the following theorem. Moreover, we investigate the iteration complexity of i-BRPD:OPF method when the upper-level objective function is linear in $y$ (i.e., $L^\Phi_{yy} = 0$).
\begin{theorem}\label{thm:proj-free}
Suppose Assumptions \ref{assump:grad-xytheta-lip}, \ref{assump:g-conditions} hold and let $\{(x_k,y_k,\theta_k)\}_{k\geq 0}$ be the sequence generated by Algorithm \ref{i-BRPD:OPF} with step-sizes $\gamma_k = \gamma>0$, $\eta_k=\eta>0$, $\alpha>0$, and $\sigma_k=\sigma\leq \frac{2}{L^\Phi_{yy}+2\mu}$ and regularization parameter $\mu>0$. Then, for any $K\geq 1$ we have that
    \begin{align}
    \frac{1}{K}\sum_{k=0}^{K-1}\mathcal{G}_{\mathcal{X}}(x_{k},y_{k}) &\leq \frac{f(x_0) - f(x^*)}{\gamma K} + \frac{B_1}{K} +\gamma B_2 D_\cX^2 + \frac{\mu}{2\gamma K}D_\cY^2, \label{eq:gap-x}\\
    \frac{1}{K}\sum_{k=0}^{K-1}\mathcal{G}_{\mathcal{Y}} (x_k,y_k)
    & \leq \frac{2\left\| y_0 - y^*_\mu(x_0) \right\|}{\sigma(1-\rho_D)K}  + \gamma B_3 D_\cX + \Big(1+\frac{2}{1-\rho_D}\Big)\frac{L_{y\theta}^\Phi}{(1-\beta)K}\left\|\theta_0 - \theta^*(x_0)\right\|  + \mu D_\mathcal{Y},\label{eq:gap-y}
\end{align}
for some constants $B_1,B_2,B_3>0$ where $B_1=\cO(\kappa_g^3+L_{y\theta}^\Phi\kappa_g^2/\mu)$, $B_2=\cO\big((\frac{\kappa_g^3}{\mu^2}+\frac{\kappa_g^4}{\mu})(L_{y\theta}^\Phi)^2+\frac{\kappa_g^2}{\mu^2}L_{y\theta}^\Phi L_{yx}^\Phi+\kappa_g^4\big)$, and $B_3 = \cO(\frac{\kappa_g}{\mu^2}+\frac{\kappa_g^2}{\mu})$.
Furthermore, if $L^\Phi_{yy} = 0$, in \eqref{eq:gap-y} we have that
\begin{align*}
     \frac{1}{K}\sum_{k=0}^{K-1}\mathcal{G}_{\mathcal{Y}} (x_k,y_k)
    & \leq \frac{3}{(1-\beta)}\frac{L_{y\theta}^\Phi}{K}\left\|\theta_0 - \theta^*(x_0)\right\| + \gamma B_3 D_\cX + \mu D_\mathcal{Y}.
\end{align*}
\end{theorem}
\begin{proof}
    From Lipschitz continuity of $\nabla f_\mu$, we have that
\begin{align}\label{eq:lip-f-mu}
    f_\mu (x_{k+1}) & \leq f_\mu (x_k) + \langle \nabla f_\mu(x_k) , x_{k+1} - x_k \rangle + \frac{1}{2}(\bL_{\mathcal L_1} + \frac{\bL_{\mathcal L_2} L_{y\mu}}{\mu}) \|{x_{k+1} - x_k}\|^2 \nonumber\\
    &=  f_\mu (x_k) + \langle \nabla f_\mu(x_k) -\nabla_x\mathcal L(x_k,y_k), x_{k+1} - x_k \rangle +\langle \nabla_x\mathcal L(x_k,y_k)-G_k^x, x_{k+1} - x_k \rangle \nonumber\\
    &\quad +\langle G_k^x, x_{k+1} - x_k \rangle+ \frac{1}{2} (\bL_{\mathcal L_1} + \frac{\bL_{\mathcal L_2} L_{y\mu}}{\mu})\|{x_{k+1} - x_k}\|^2,
\end{align}
where the last equality follows from adding and subtracting the terms $\langle G_k^x+\nabla_x\mathcal L(x_k,y_k), x_{k+1} - x_k \rangle$ to RHS. Recall that $x_{k+1}=\gamma s_k+(1-\gamma)x_k$. Moreover, we define $s'_k\triangleq \argmax_{s\in X} \langle\nabla_x\mathcal L(x_k,y_k),x_k-s\rangle$ then, we have that
\begin{align}\label{eq:inner-prod-Gk}
\langle G_k^x, x_{k+1} - x_k \rangle&=\gamma\langle G_k^x, s_k - x_k \rangle \nonumber\\
&\leq \gamma \langle G_k^x, s'_k - x_k \rangle \nonumber\\
&=\gamma \langle \nabla_x\mathcal L(x_k,y_k), s'_k - x_k \rangle+\gamma \langle G_k^x-\nabla_x\mathcal L(x_k,y_k), s'_k - x_k \rangle.
\end{align}
Next, using \eqref{eq:inner-prod-Gk} within \eqref{eq:lip-f-mu}, rearranging the terms, and using Cauchy Schwartz inequality implies that
\begin{align*} 
    \gamma\langle \nabla_x\mathcal L(x_k,y_k), x_k - s'_k \rangle&\leq f_\mu (x_k) - f_\mu (x_{k+1})+ \|\nabla f_\mu(x_k) -\nabla_x\mathcal L(x_k,y_k)\| \|x_{k+1} - x_k\|  \nonumber\\
    &\quad +\gamma\|\nabla_x \mathcal L(x_k,y_k)-G_k^x\| \|s_k-s'_k\| \nonumber\\
    & \quad + \frac{1}{2} (\bL_{\mathcal L_1} + \frac{\bL_{\mathcal L_2} L_{y\mu}}{\mu})\|{x_{k+1} - x_k}\|^2.
\end{align*}
From Lemma \ref{lem:grad-fmu-lip}, we know that $\nabla f_\mu(x) = \nabla_x \mathcal{L} (x,y^*_\mu(x))$. Therefore, using Lipschitz continuity of $\cL(x,\cdot)$, $\norm{s'_k-x_k}\leq D_\cX$, and $\|{x_{k+1} - x_k}\|= \gamma \norm{s_k-x_k}$, we obtain
\begin{align*} 
    \gamma\langle \nabla_x\mathcal L(x_k,y_k), x_k - s'_k \rangle&\leq f_\mu (x_k) - f_\mu (x_{k+1})+ \bL_{\cL_2}\gamma\norm{y_k-y_\mu^*(x_k)} \norm{s_k-x_k}  \nonumber\\
    &\quad +\gamma D_\cX\|\nabla_x \mathcal L(x_k,y_k)-G_k^x\|   + \frac{1}{2} (\bL_{\mathcal L_1} + \frac{\bL_{\mathcal L_2} L_{y\mu}}{\mu})\gamma^2\norm{s_k-x_k}^2.
\end{align*}%
Summing the above inequality over $k=0$ to $K-1$,
dividing both sides by $\gamma K$,
and recalling $\mathcal{G}_{\mathcal{X}}(x_k,y_k) = \sup_{s\in \cX} \langle \nabla_x\mathcal L(x_k,y_k), x_k - s \rangle$ imply that for any $K\geq 1$,
\begin{align*}
    \frac{1}{K} \sum_{k\in \mathcal{K}} \mathcal{G}_{\mathcal{X}}(x_k,y_k) &\leq \frac{(f_\mu(x_0) - f_\mu(x_K))}{\gamma K} + \frac{\bL_{\mathcal L_2}}{K} \sum_{k =0}^{K-1}\norm{s_k-x_k}\left\| y_k - y^*_\mu (x_k) \right\| \nonumber\\
    & \quad + \frac{D_{\mathcal{X}}}{K} \sum_{k=0}^{K-1}\|\nabla_x \mathcal L(x_k,y_k)-G_k^x\| + \gamma (\bL_{\mathcal L_1} + \frac{\bL_{\mathcal L_2} L_{y\mu}}{\mu})\frac{1}{K}\sum_{k=0}^{K-1}\norm{s_k-x_k}^2.
\end{align*}
By noting that $\norm{s_k-x_k}\leq D_\cX$ due to the boundedness of set $\cX$, we can apply Lemmas \ref{lem: y-ymu-bound}, \ref{lem: err-gradxL} to the above inequality to obtain
\begin{align*}
    \frac{1}{K} \sum_{k=0}^{K-1} \mathcal{G}_{\mathcal{X}}(x_k,y_k) &\leq \frac{(f_\mu(x_0) - f_\mu(x_K))}{\gamma K} + \frac{\bL_{\mathcal L_2}}{K} \left(\Gamma_1+\Gamma_2 D_\cX^2 K\right) \nonumber\\
    & \quad + \frac{\Gamma_4+2\bC_{v2}\Gamma_1}{K} D_\cX+(\bC_{v2}\Gamma_3+\Gamma_5)D_\cX^2 + \gamma \left(\bL_{\mathcal L_1} + \frac{\bL_{\mathcal L_2} L_{y\mu}}{\mu}\right) D_\cX^2.
\end{align*}
From the definition of $f_\mu$ and boundedness of $\cY$, we have $f_\mu(x_0)-f_\mu(x_K)\leq f(x_0)-f(x^*)+\frac{\mu}{2} D_\cY^2$. Therefore, for any $K>2$
\begin{align*}
    \frac{1}{K} \sum_{k=0}^{K-1} \mathcal{G}_{\mathcal{X}}(x_k,y_k) &\leq \frac{(f(x_0) - f(x^*))}{\gamma K} + \frac{1}{K} \overbrace{\Big(\bL_{\mathcal L_2}\Gamma_1+D_\cX(\Gamma_4+2\bC_{v2}\Gamma_1) \Big)}^{\triangleq B_1} \nonumber\\
    & \quad + \underbrace{\left(\frac{1}{\gamma}(\bL_{\cL_2}\Gamma_2+\bC_{v2}\Gamma_3+\Gamma_5) + \bL_{\mathcal L_1} + \frac{\bL_{\mathcal L_2} L_{y\mu}}{\mu}\right)}_{\triangleq B_2} \gamma D_\cX^2 + \frac{\mu}{2\gamma K}D_\cY^2.
\end{align*}
Finally, recall that $\Gamma_1=\cO(\frac{\kappa_g}{\mu}L_{y\theta}^\Phi)$, $\Gamma_2=\cO(\frac{\kappa_gL_{y\theta}^\Phi+L_{y x}^\Phi}{\mu^2}\gamma +\frac{\kappa_g^2L_{y\theta}^\Phi}{\mu}\gamma)$, $\Gamma_3=\cO(\frac{\kappa_g^2L_{y\theta}^\Phi+\kappa_g L_{yx}^\Phi}{\mu^2}\gamma+\frac{\kappa_g^3L_{y\theta}^\Phi}{\mu}\gamma)$, $\Gamma_4+2\bC_{v2}\Gamma_1=\cO(\kappa_g^3+L_{y\theta}^\Phi \kappa_g^2/\mu)$, $\bC_{v2}\Gamma_3+\Gamma_5=\cO((\frac{\kappa_g^3}{\mu^2}+\frac{\kappa_g^4}{\mu})(L_{y\theta}^\Phi)^2\gamma+\frac{\kappa_g^2}{\mu^2}L_{y\theta}^\Phi L_{yx}^\Phi\gamma+\kappa_g^4\gamma)$, $\bL_{\cL_1}=\cO(\kappa_g^3)$, and $\bL_{\cL_2}=\cO(\kappa_g)$. Putting the pieces together, the desired result in \eqref{eq:gap-x} can be obtained. 

Next, we obtain an upper-bound for the dual gap function $\mathcal{G}_{\mathcal{Y}}(x_k,y_k) = \frac{1}{\sigma}\left\| y_k - \mathcal{P}_{\mathcal{Y}}(y_k + \sigma\nabla_y \mathcal{L}(x_k,y_k))\right\|$.
In particular, using the triangle inequality and by adding and subtracting $\mathcal{P}_{\mathcal{Y}} \big( y_k + \sigma G_k^y\big)$, we have that for any $k\geq0$,
\begin{align*}
    \mathcal{G}_{\mathcal{Y}} (x_k,y_k) \leq \frac{1}{\sigma}\big(\left\| y_k - \mathcal{P}_{\mathcal{Y}} \big( y_k + \sigma G_k^y\big) \right\| + \left\| \mathcal{P}_{\mathcal{Y}} \big( y_k + \sigma G_k^y\big) - \mathcal{P}_{\mathcal{Y}}(y_k + \sigma \nabla_y \mathcal{L} (x_k,y_k)) \right\|\big).
\end{align*}
Considering non-expansivity of the projection mapping, we have
\begin{align*} 
    \mathcal{G}_{\mathcal{Y}} (x_k,y_k) \leq \frac{1}{\sigma}(\left\| y_k - y_{k+1} \right\| + \sigma \left\| G_k^y - \nabla_y \mathcal{L}(x_k,y_k) \right\|).
\end{align*}
Based on the conclusion of Lemma \ref{lem: err-gradyL}, we can get
\begin{align*}
    \mathcal{G}_{\mathcal{Y}} (x_k,y_k) \leq \frac{1}{\sigma}\left\| y_k - y_{k+1} \right\| + L_{y\theta}^\Phi\beta^k \left\|\theta_0 - \theta^*(x_0)\right\| + \gamma L_{y\theta}^\Phi\bL_\theta \sum_{i=0}^{k-1}\beta^{k-i}\norm{s_i-x_i} + \mu D_\mathcal{Y}.
\end{align*}
Moreover, using Lemma \ref{lem: y-bound} and rearragning the terms, we conclude that for any $k\geq 0$, 
\begin{align}\label{eq: Gap-y-final}
    \mathcal{G}_{\mathcal{Y}} (x_k,y_k) &\leq 
    \frac{2}{\sigma}\rho_D^k \left\| y_0 - y^*_\mu(x_0) \right\| + \frac{2\gamma L_{y\mu}}{\sigma\mu} \sum_{i=0}^{k-1} \rho_D^{k-i} \left\| s_i - x_i \right\| + 2L_{y\theta}^\Phi \left\|\theta_0 - \theta^*(x_0)\right\|\sum_{i=1}^{k}\rho^{k-i}_D\beta^i \nonumber\\
    & \quad + 2\gamma L_{y\theta}^\Phi \bL_\theta\sum_{i=1}^k\rho^{k-i}_D  \sum_{j=0}^{i-1} \beta^{i-j} \left\| s_j - x_j \right\|+ \frac{L_{y\mu}}{\mu\sigma}\gamma \norm{s_k-x_k}\nonumber\\
    &\quad + L_{y\theta}^\Phi\beta^k \left\|\theta_0 - \theta^*(x_0)\right\| + \gamma L_{y\theta}^\Phi\bL_\theta \sum_{i=0}^{k-1}\beta^{k-i}\norm{s_k-x_k} + \mu D_\mathcal{Y}.
\end{align}
Summing the above inequality over $k=0$ to $K-1$, using Lemma \ref{lem:sum-sum-linear-comb} together with boundedness of set $\cX$ lead to
\begin{align}\label{eq: Gap-y-final-av}
    \frac{1}{K}\sum_{k=0}^{K-1}\mathcal{G}_{\mathcal{Y}} (x_k,y_k)
    & \leq \frac{2\left\| y_0 - y^*_\mu(x_0) \right\|}{\sigma(1-\rho_D)K}  + \frac{2L_{y\mu}\rho_D\gamma D_{\mathcal{X}}}{\mu(1-\rho_D)\sigma} +  \frac{2 L_{y\theta}^\Phi\bL_\theta \beta\gamma D_{\mathcal{X}}}{(1-\beta)(1-\rho_D)} + \frac{L_{y\mu}}{\mu \sigma}\gamma D_{\mathcal{X}} \nonumber\\
    & \quad + \frac{L_{y\theta}^\Phi\bL_\theta \beta\gamma D_{\mathcal{X}}}{1-\beta} + \eyh{\Big(1+\frac{2}{1-\rho_D}\Big)}\frac{L_{y\theta}^\Phi}{(1-\beta)K}\left\|\theta_0 - \theta^*(x_0)\right\|  + \mu D_\mathcal{Y},\nonumber\\
    & = \underbrace{\Big(\frac{2L_{y\mu}\rho_D}{\mu(1-\rho_D)\sigma} +  \frac{2 L_{y\theta}^\Phi\bL_\theta \beta}{(1-\beta)(1-\rho_D)} + \frac{L_{y\mu}}{\mu \sigma} + \frac{L_{y\theta}^\Phi\bL_\theta \beta}{1-\beta}\Big)}_{\triangleq B_3}\gamma D_{\mathcal{X}} \nonumber\\
    & \quad + \frac{2\left\| y_0 - y^*_\mu(x_0) \right\|}{\sigma(1-\rho_D)K} + \Big(1+\frac{2}{1-\rho_D}\Big)\frac{L_{y\theta}^\Phi}{(1-\beta)K}\left\|\theta_0 - \theta^*(x_0)\right\|  + \mu D_\mathcal{Y}.
\end{align}
Therefore, the desired result can be concluded by noting that $1-\rho_D=\cO(\mu)$, $1-\beta=\cO(1/\kappa_g)$, and $\bL_\theta=\cO(\kappa_g)$. Furthermore, if $L^\Phi_{yy} = 0$, then $\rho_D = 0$, and the bound obtained in \eqref{eq: Gap-y-final-av} simplifies to
\begin{align*}
    \frac{1}{K}\sum_{k=0}^{K-1}\mathcal{G}_{\mathcal{Y}} (x_k,y_k)
    & \leq \frac{3}{(1-\beta)}\frac{L_{y\theta}^\Phi}{K}\left\|\theta_0 - \theta^*(x_0)\right\| + \frac{3 L^\Phi_{y \theta}\bL_\theta}{(1-\beta)}\gamma D_{\cX}+ L_{y \mu}\gamma D_{\cX}+ \mu D_\mathcal{Y}.
\end{align*}
\end{proof}
\begin{corollary}\label{cor:proj-free}
Under the premises of Theorem \ref{thm:proj-free}, let $\mu=\mathcal O(\epsilon)$ and $\gamma=\mathcal O(\epsilon^3)$, then for any $K\geq 1$, there exists $t\in \{0,\hdots,K-1\}$ such that $(x_t,y_t)\in \mathcal{X}\times \mathcal{Y}$ satisfies  $\cG_\cZ(x_t,y_t)\leq \epsilon$ within $\cO(\frac{\kappa_g^3}{\epsilon^4}+\frac{\kappa_g^5}{\epsilon^2})$ iterations.
\end{corollary}
\begin{proof}
    Summing up the results on the primal and dual gap functions in \eqref{eq:gap-x} and \eqref{eq:gap-y} and defining $k^*\triangleq \argmin_{k=0}^{K-1}\cG_\cZ(x_k,y_k) $
implies that $\mathcal{G}_{\mathcal{Z}}(x_{k^*},y_{k^*}) \leq \cO(\frac{1}{\gamma K} +\gamma B_2 +\mu)$ where $B_2=\cO(\frac{\kappa_g^3}{\mu^2}+\frac{\kappa_g^4}{\mu})$. Selecting $\mu=\kappa_g^{3/4}/K^{1/4}$ and $\gamma = 1/(\kappa_g K)^{3/4}$ implies that $\cG_\cZ(x_{k^*},y_{k^*})\leq \cO(\frac{\kappa_g^{3/4}}{K^{1/4}}+\frac{\kappa_g^{2.5}}{\sqrt{K}})$. Therefore, to achieve $\cG_\cZ(x_{k^*},y_{k^*})\leq \epsilon$, Algorithm \ref{i-BRPD:OPF} requires $K=\cO(\frac{\kappa_g^3}{\epsilon^4}+\frac{\kappa_g^5}{\epsilon^2})$ iterations.
\end{proof}
\begin{corollary}\label{cor:proj-free-linear_y}
Suppose the premises of Theorem \ref{thm:proj-free} holds and we further assume that $L_{yy}^\Phi=0$. Let $\mu=\cO(\epsilon)$, $\gamma = \cO(\epsilon^3)$, then for any $K\geq 1$, there exists $t\in \{0,\hdots,K-1\}$ such that $(x_t,y_t)\in \mathcal{X}\times \mathcal{Y}$ satisfies  $\cG_\cZ(x_t,y_t)\leq \epsilon$ within $\cO(\frac{\kappa_g^4}{\epsilon^3})$ iterations.
\end{corollary}
\begin{proof}
    Recall that assuming $L^\Phi_{yy} = 0$, we have $\rho_D = 0$, leading to simplification of the bounds in Lemmas \ref{lem: lower-y-ymu}, \ref{lem: y-ymu-bound}, and \ref{lem: y-bound}. In particular, we derive $\Gamma_1 = \frac{L^\Phi_{y\theta}}{\mu (1-\beta)}\|\theta_0 - \theta^*(x_0)\|D_{\mathcal{X}} = \mathcal{O}(\frac{\kappa_g L^\Phi_{y\theta}}{\mu})$, $\Gamma_2 = \frac{L^\Phi_{y\theta} \gamma \bL_\theta \beta}{\mu (1-\beta)} = \cO(\frac{\kappa_g^2 L^\Phi_{y \theta}}{\mu}\gamma)$ and $\Gamma_3= \frac{2\Gamma_2}{1-\rho}+\frac{L_{y\mu}\gamma}{(1-\rho)\mu} = \cO(\frac{\kappa_g^3 L^\Phi_{y \theta}}{\mu}\gamma + \frac{\kappa_g^2 L^\Phi_{y \theta} + \kappa_g L^\Phi_{y x}}{\mu}\gamma)$, simplifying the bound on the primal and dual gap functions obtained in Theorem \ref{thm:proj-free} with the following constants $B_1= \cO(\kappa_g^3+L_{y\theta}^\Phi\kappa_g^2/\mu)$, $B_2=\cO(\frac{\kappa_g^4}{\mu}(L^\Phi_{y \theta})^2 + \kappa_g^4)$, and $B_3= \cO(\kappa_g^2 L^\Phi_{y \theta} \gamma)$. Using a similar argument as in the proof of Corollary \ref{cor:proj-free}, and noting that $\cG_\cZ(x_{k^*},y_{k^*})\leq \cO(\frac{1}{\gamma K}+\gamma B_2+\mu)$, choosing $\mu = \kappa_g^2/K^{1/3}$ and $\gamma = 1/(\kappa_gK)^{2/3}$, it follows that $\cG_\cZ(x_{k^*},y_{k^*})\leq \cO(\frac{\kappa_g^{4/3}}{K^{1/3}})$. Therefore, to achieve $\cG_\cZ(x_{k^*},y_{k^*})\leq \epsilon$, Algorithm \ref{i-BRPD:OPF} requires $K=\cO(\frac{\kappa_g^4}{\epsilon^3})$ iterations.
\end{proof}
\begin{remark}\label{remark:proj-free}
    Our results in Theorem \ref{thm:proj-free} represent a new bound for the constrained nonconvex-concave saddle point problem subject to a strongly convex problem in the lower-level setting. Specifically, it introduces an upper bound on a gap function for the primal part of Problem \eqref{eq:main-prob}, which corresponds to the linear minimization oracle (LMO), and an upper bound on a gap function for the dual part of Problem \eqref{eq:main-prob} for the iterates generated by Algorithm \ref{i-BRPD:OPF}. To the best of our knowledge, this is the only one-sided projection-free algorithm for solving the general non-bilinear SPB problem. Additionally, this method reduces the computational cost compared to projection-based algorithms for certain constraint sets $\mathcal{X}$, such as the nuclear-norm ball, Schatten-norm ball, and certain polyhedron sets \cite{jaggi2013revisiting}.
\end{remark}
\begin{remark}\label{remark:comp-bilevel-opt}
    One can replace the projection operator in the $y$-update with a proximal map of a proper lower semi-continuous convex function $h$ with a bounded domain and Algorithm \ref{i-BRPD:OPF} and convergence results in Corollaries \ref{cor:proj-free} and \ref{cor:proj-free-linear_y} will still hold. Moreover, it is worth noting that in the case $L^\Phi_{yy} = 0$, we can solve a bilevel optimization problem with a composite objective function of the form: $\min_x h(\Tilde{\Phi}(x,\theta(x))$  s.t. $\theta(x) \in \argmin_\theta g(x,\theta)$ where $\Tilde{\Phi}$ is a smooth map. In this case, $\Phi(x,y) = \langle \Tilde{\Phi}(x,\theta(x),y \rangle$, and based on Corollary \ref{cor:proj-free-linear_y}, the complexity of achieving $\epsilon$-stationary solution is $\cO(\frac{\kappa_g^4}{\epsilon^3})$, which constitutes a new result in the context of bilevel optimization.
\end{remark}
\subsection{Convergence Rate of Algorithm \ref{i-BRPD:FP}}\label{subsec:conv2}
In this subsection, we establish upper bounds for the primal and dual gap functions for the iterates generated by Algorithm \ref{i-BRPD:FP}. Additionally, we explore the iteration complexity of the i-BRPD:FP method when the upper-level objective function exhibits linearity in $y$ (i.e., $L^\Phi_{yy} = 0$).
\begin{theorem}\label{thm:proj}
Suppose Assumptions \ref{assump:grad-xytheta-lip}, \ref{assump:g-conditions} hold and let $\{(x_k,y_k,\theta_k)\}_{k\geq 0}$ be the sequence generated by Algorithm \ref{i-BRPD:FP} with step-sizes $\gamma_k = \gamma>0$, $\eta_k=\eta>0$, $\alpha>0$, $\tau_k=\tau>0$, and $\sigma_k=\sigma\leq \frac{2}{L^\Phi_{yy}+2\mu}$ and regularization parameter $\mu>0$. Then, for any \mma{$K\geq 1$}, we have that
    \begin{align*}
        \frac{1}{\tau^2 K}\sum_{k=0}^{K-1}\norm{\ey{\tilde{s}_k}-x_k}^2 
        &\leq  \frac{2(f(x_0) - f(x^*))+\mu D_\cY}{\tau \gamma K}+\ey{\frac{2\bL_{\cL_2}\Gamma_1+4\tau M\Gamma_4}{\tau K}+4\tau\gamma  M^2 B_4},\\
        \frac{1}{K}\sum_{k=0}^{K-1}\mathcal{G}_{\mathcal{Y}} (x_k,y_k)
        & \leq \frac{2\left\| y_0 - y^*_\mu(x_0) \right\|}{\sigma(1-\rho_D)K}  + \gamma B_3 D_\cX + \Big(1+\frac{2}{1-\rho_D}\Big)\frac{L_{y\theta}^\Phi}{(1-\beta)K}\left\|\theta_0 - \theta^*(x_0)\right\|  + \mu D_\mathcal{Y},
    \end{align*}
    for some constant $B_3,B_4>0$, where $B_3$ is defined in Theorem \ref{thm:proj-free} and $B_4=\cO(\kappa_g^3/\mu)$.
\end{theorem}
\begin{proof}
    Similar to \eqref{eq:lip-f-mu}, from Lipschitz continuity of $\nabla f_\mu$ we have that
\begin{align}\label{eq:lip-f-mu-2}
    f_\mu (x_{k+1}) & \leq   f_\mu (x_k) + \langle \nabla f_\mu(x_k) -\nabla_x\mathcal L(x_k,y_k), x_{k+1} - x_k \rangle +\langle \nabla_x\mathcal L(x_k,y_k), x_{k+1} - x_k \rangle \nonumber\\
    &\quad + \frac{1}{2} (\bL_{\mathcal L_1} + \frac{\bL_{\mathcal L_2} L_{y\mu}}{\mu})\|{x_{k+1} - x_k}\|^2\nonumber\\
    &=f_\mu (x_k) + \langle \nabla f_\mu(x_k) -\nabla_x\mathcal L(x_k,y_k), x_{k+1} - x_k \rangle  +\gamma\langle G_k^x, s_k - x_k \rangle\nonumber\\
    &\quad +\gamma\langle \nabla_x\mathcal L(x_k,y_k)-G_k^x, s_k - x_k \rangle+ \frac{1}{2} (\bL_{\mathcal L_1} + \frac{\bL_{\mathcal L_2} L_{y\mu}}{\mu})\|{x_{k+1} - x_k}\|^2,
\end{align}
where in the last inequality we add and subtract $\gamma\fprod{G_k,s_k-x_k}$. Recall the update of $s_k=\cP_\cX(x_k-\tau G_k^x)$ and $x_{k+1}=\gamma s_k+(1-\gamma)x_k$. The optimality condition of this step implies that
\begin{align}\label{eq:opt-proj}
    \gamma\fprod{G_k^x,s_k-x_k}=\fprod{G_k^x,x_{k+1}-x_k}\leq -\frac{\gamma}{\tau}\norm{s_k-x_k}^2.
\end{align} 
\ey{Moreover, let us define $\tilde{s}_k = \cP_{\cX}(x_k-\tau \nabla_x \cL(x_k,y_k))$ which implies that $\cG_\cX(x_k,y_k)=\frac{1}{\tau}\|\tilde{s}_k-x_k\|$. Considering $\gamma\fprod{G_k^x,s_k-x_k}$ once again and adding and subtracting $\fprod{\nabla_x\cL(x_k,y_k),\tilde{s}_k-s_k}$ we obtain 
\begin{align}\label{eq:proj-s-tilde}
    \gamma\fprod{G_k^x,s_k-x_k}&=\gamma\fprod{G_k^x,s_k-\tilde{s}_k}+\gamma\fprod{G_k^x-\nabla_x\cL(x_k,y_k),\tilde{s}_k-x_k}+\gamma\fprod{\nabla_x\cL(x_k,y_k),\tilde{s}_k-x_k}\nonumber\\
    &\leq \gamma\left(\|G_k^x\|\|s_k-\tilde{s}_k\|+\|G_k^x-\nabla_x\cL(x_k,y_k)\|\|\tilde{s}_k-x_k\|-\frac{1}{\tau}\|\tilde{s}_k-x_k\|^2\right)\nonumber\\
    &\leq \gamma\left(\tau(\|G_k^x\|+\|\nabla_x\cL(x_k,y_k)\|)\|G_k^x-\nabla_x\cL(x_k,y_k)\|-\frac{1}{\tau}\|\tilde{s}_k-x_k\|^2\right),
\end{align}
where in the first inequality we used Cauchy-Schwartz inequality and the optimality condition of $\tilde{s}_k$, and the last inequality follows from nonexpansivity of the projection operator and that $x_k=\cP_\cX(x_k)$.
Combining \eqref{eq:opt-proj} and \eqref{eq:proj-s-tilde} we arrive at 
\begin{equation*}
    \gamma\fprod{G_k^x,s_k-x_k}\leq \frac{\gamma}{2}\left(\tau(\|G_k^x\|+\|\nabla_x\cL(x_k,y_k)\|)\|G_k^x-\nabla_x\cL(x_k,y_k)\|-\frac{1}{\tau}\|\tilde{s}_k-x_k\|^2-\frac{1}{\tau}\|{s}_k-x_k\|^2\right).
\end{equation*}
Since the domains $\mathcal{X}$ and $\mathcal{Y}$ are bounded, and $\cL$ is continuously differentiable we conclude that $\|\nabla_x \mathcal{L}(x_k, y_k)\|$ is bounded. Furthermore, as shown in Lemmas \ref{lem:lower_GD_C} and \ref{lem:lower-w-v}, the sequences $\{\theta_k\}_{k\geq 0}$ and $\{w_k\}_{k\geq 0}$ are bounded, which together with continuous differentiability of functions $\Phi$ and $g$ in turn implies the uniform boundedness of $\|G_k^x\|$. Let us define $M\triangleq \sup_{k\geq 0}\max\{\|G_k^x\|,\|\nabla_x \cL(x_k,y_k)\|\}<+\infty$. Therefore, plugging the above inequality in \eqref{eq:lip-f-mu-2} and rearranging the terms we obtain
\begin{align*}
    \frac{\gamma}{2\tau}\norm{\tilde{s}_k-x_k}^2 &\leq  f_\mu (x_k) - f_\mu (x_{k+1})+\gamma \bL_{\cL_2}\norm{y_k-y_\mu^*(x_k)}\norm{s_k-x_k} +2\gamma\tau M \|\nabla_x \mathcal L(x_k,y_k)-G_k^x\|  \nonumber\\
    &\quad  + \left(\frac{1}{2} (\bL_{\mathcal L_1} + \frac{\bL_{\mathcal L_2} L_{y\mu}}{\mu})\gamma^2-\frac{\gamma}{2\tau}\right)\norm{s_k-x_k}^2.
\end{align*}}%
\mma{Summing the above inequality over $k=0$ to $K-1$, dividing both sides by $\gamma K$, and using the second part of Lemma \eqref{lem: y-ymu-bound} we obtain
\begin{align}\label{eq:rate-initial-x}
    \frac{1}{2\tau K}\sum_{k=0}^{K-1}\|\tilde{s}_k-x_k\|^2&\leq \frac{f_\mu (x_0) - f_\mu (x_K)}{\gamma K}+ \frac{\bL_{\cL_2}\Gamma_1}{K}+\underbrace{\frac{2\tau M}{K}\sum_{k=0}^{K-1}\|\nabla_x \mathcal L(x_k,y_k)-G_k^x\|}_{(*)}\nonumber\\
    &\quad +\left(\bL_{\cL_2}\Gamma_2+\frac{1}{2} (\bL_{\mathcal L_1} + \frac{\bL_{\mathcal L_2} L_{y\mu}}{\mu})\gamma-\frac{1}{2\tau }\right)\frac{1}{K}\sum_{k=0}^{K-1}\norm{s_k-x_k}^2.
\end{align}}
\ey{Next, we establish an upper bound for the term $(*)$. To this end, we begin by obtaining a bound for the consecutive iterates of $y_k$ as follows,
\begin{align*}
    \|y_{k+1}-y_k\|&\leq \|y_k-\sigma\nabla_y\cL_\mu(x_k,y_k)-y_{k-1}+\sigma\nabla_y\cL_\mu(x_{k-1},y_{k-1})\|\\
    &\leq \sigma\|\nabla_y\cL_\mu(x_k,y_k)-\nabla_y\cL_\mu(x_{k-1},y_k)\|+\|y_k-\sigma\nabla_y\cL_\mu(x_{k-1},y_k)-y_{k-1}+\sigma\nabla_y\cL_\mu(x_{k-1},y_{k-1})\|\\
    &\leq \sigma \bL_{\cL_2}\|x_k-x_{k-1}\|+\rho_D\|y_k-y_{k-1}\|,
\end{align*}
where in the first inequality we used the nonexpansivity of the projection operator, the second inequality is followed by the triangle inequality, and the last inequality is the result of Lipschitz continuity of $\nabla_y\cL_\mu$ and a similar argument as in the proof of Lemma \ref{lem: lower-y-ymu}. Continuing the above recursive relation and using the facts that $\|x_{k+1}-x_k\|\leq \gamma\tau M$ and $y_{-1}=y_0$, we obtain 
\begin{equation*}
    \|y_{k+1}-y_k\|\leq \frac{\sigma \gamma\tau}{1-\rho_D}  \bL_{\cL_2} M. 
\end{equation*}
Therefore, considering the first part of Lemma \ref{lem: err-gradxL}, using the above inequality, and that $\|s_k-x_k\|\leq \tau M$ for any $k\geq 0$, we obtain
\begin{align*}
    \sum_{k=0}^{K-1}\|\nabla_x \mathcal L(x_k,y_k)-G_k^x\|&\leq \Gamma_4+\tau M\Gamma_5 K +C_{\theta x}^g\bC_{v2}\frac{\sigma \gamma \tau \bL_{\cL_2}}{(1-\rho)(1-\rho_D)}M K.
\end{align*}}%
\mma{Therefore, using the above inequality within \eqref{eq:rate-initial-x}, the step-size condition $\tau\leq 1/(2\bL_{\cL_2}\Gamma_2+ (\bL_{\mathcal L_1} + \frac{2\bL_{\mathcal L_2} L_{y\mu}}{\mu})\gamma)=\cO(1/(\kappa_g^3\gamma))$, and multiplying both sides by $2/\tau$ we obtain
\begin{equation*}
    \frac{1}{\tau^2 K}\sum_{k=0}^{K-1}\norm{\ey{\tilde{s}_k}-x_k}^2 
        \leq\frac{2(f(x_0) - f(x^*))+\mu D_\cY}{\tau \gamma K}+ \frac{2\bL_{\cL_2}\Gamma_1+4\tau M\Gamma_4}{\tau K}+4\tau\gamma  M^2\underbrace{\Big(\frac{\Gamma_5}{\gamma}+\frac{\sigma C_{\theta x}^g\bC_{v2}\bL_{\cL_2}}{(1-\rho)(1-\rho_D)}\Big)}_{\triangleq B_4},
\end{equation*}}
\mma{which completes the first part of the proof by noting that $1-\rho=\Omega(1/\kappa_g)$, $1-\rho_D=\Omega(\mu)$, $\bL_{\cL_2}=\cO(\kappa_g)$, $\bC_{v2}=\cO(\kappa_g)$ and $\Gamma_5=\cO(\gamma\kappa_g^3)$ imply that \ey{$B_4=\cO(\kappa_g^3/\mu)$}}.

Next, using the same line of proof as in Theorem \ref{thm:proj-free}, we can derive a bound for the dual gap function as in \eqref{eq: Gap-y-final}. Therefore, using boundedness of constraint set $\cX$  and by noting that $\sum_{i=1}^k \rho_D^{k-i}\beta^i\leq (\frac{\rho_D^{k/2}}{1-\beta}+\frac{\beta^{k/2}}{1-\rho_D})\beta$ for any $k\geq 2$, we conclude that
\begin{align}\label{eq:dual-gap}
    \mathcal{G}_{\mathcal{Y}} (x_k,y_k) &\leq 
    \frac{2}{\sigma}\rho_D^k \left\| y_0 - y^*_\mu(x_0) \right\| + \frac{2\gamma L_{y\mu}}{\sigma\mu(1-\rho_D)}D_\cX + 2L_{y\theta}^\Phi \beta\left(\frac{\rho_D^{k/2}}{1-\beta}+\frac{\beta^{k/2}}{1-\rho_D}\right) \left\|\theta_0 - \theta^*(x_0)\right\|\nonumber\\
    & \quad + \frac{2\gamma \beta L_{y\theta}^\Phi \bL_\theta}{(1-\rho_D)(1-\beta)}D_\cX+ \frac{L_{y\mu}}{\mu\sigma}\gamma D_\cX + L_{y\theta}^\Phi\beta^k \left\|\theta_0 - \theta^*(x_0)\right\| + \frac{\gamma L_{y\theta}^\Phi\bL_\theta}{1-\beta} D_\cX + \mu D_\mathcal{Y}.
\end{align}

\end{proof}
\begin{corollary}\label{cor:proj}
Under the premises of Theorem \ref{thm:proj}, let $\mu=\mathcal O(\epsilon)$, $\gamma=\mathcal O(\epsilon^3)$ and $\tau=\cO(1/\kappa_g^3)
$, then for any $K\geq 2\bar K=\cO(1/\mu\log(1/\mu))$, there exists $t\in \{\bar K,\hdots,K-1\}$ such that $(x_t,y_t)\in \mathcal{X}\times \mathcal{Y}$ satisfies  $\cG_\cZ(x_t,y_t)\leq \epsilon$ within $\cO(\frac{\kappa_g^3}{\epsilon^5})$ 
iterations.
\end{corollary}
\begin{proof}
    Considering the result 
    of Theorem \ref{thm:proj} on the primal gap function and multiplying both sides by $K/(K-\bar K)$ for some $0<\bar K\leq K/2$. Then, we have that for any $K\geq 2$,
\begin{align*}
    \frac{1}{\tau^2 (K-\bar K)}\sum_{k=\bar K}^{K-1}\norm{\tilde{s}_k-x_k}^2 &\leq \frac{1}{\tau^2 (K-\bar K)}\sum_{k=0}^{K-1}\norm{s_k-x_k}^2 \nonumber\\
    & \leq  \frac{2(f(x_0) - f(x^*))+\mu D_\cY^2}{\tau\gamma (K-\bar K)}+ 
    \frac{2\bL_{\cL_2}\Gamma_1+4\tau M\Gamma_4}{\tau (K-\bar{K})}+8\tau\gamma  M^2 B_4. 
\end{align*}
Let $k^*= \argmin_{k=\bar K}^{K-1} \norm{\tilde{s}_k-x_k}$. Lower-bounding the LHS and taking the square root from both sides of the above inequality, we obtain
\begin{align*}
    \cG_\cX(x_{k^*},y_{k^*})=\frac{1}{\tau} \norm{\tilde{s}_{k^*}-x_{k^*}} 
    & \leq  \sqrt{\frac{2(f(x_0) - f(x^*))+\mu D_\cY^2}{\gamma\tau (K-\bar K)}}+    
    \sqrt{\frac{2\bL_{\cL_2}\Gamma_1+4\tau M\Gamma_4}{\tau (K-\bar{K})}}+\sqrt{8\tau\gamma  M^2 B_4} 
    \nonumber\\
    &= \cO\left(\frac{1}{\sqrt{\gamma\tau K}}+\sqrt{\frac{\gamma\tau \kappa_g^3}{\mu}}\right).
\end{align*}

Next, we focus on the dual gap function in \eqref{eq:dual-gap}. Let us define $\bar K\triangleq \max\{\frac{2}{1-\rho_D}\log(\frac{1}{(1-\beta)\mu}),\frac{2}{1-\beta}\log(\frac{1}{(1-\rho_D)\mu})\}$. Then, for any $k\geq \bar K$, we have that $\max\{\rho_D^{k/2}/(1-\beta),\beta^{k/2}/(1-\rho_D)\}\leq \mu$. Therefore, for any $k\geq \bar K$ the RHS in \eqref{eq:dual-gap} can be further bounded as follows
\begin{align*}
    \mathcal{G}_{\mathcal{Y}} (x_k,y_k) &\leq 
    \frac{2}{\sigma}\mu^2(1-\beta)^2 \left\| y_0 - y^*_\mu(x_0) \right\| + \frac{2\gamma L_{y\mu}}{\sigma\mu(1-\rho_D)}D_\cX + 2L_{y\theta}^\Phi \beta\mu  \left\|\theta_0 - \theta^*(x_0)\right\|\nonumber\\
    & \quad + \frac{2\gamma \beta L_{y\theta}^\Phi \bL_\theta}{(1-\rho_D)(1-\beta)}D_\cX+ \frac{L_{y\mu}}{\mu\sigma}\gamma D_\cX + L_{y\theta}^\Phi\mu^2(1-\rho_D)^2 \left\|\theta_0 - \theta^*(x_0)\right\| + \frac{\gamma L_{y\theta}^\Phi\bL_\theta}{1-\beta} D_\cX + \mu D_\mathcal{Y}\nonumber\\
    &=\cO(\frac{\gamma\kappa_g}{\mu^2}+\frac{\gamma\kappa_g^2}{\mu}+\mu),
\end{align*}
where the last equality follows from $1-\rho_D=\Omega(\mu)$, $1-\beta=\Omega(1/\kappa_g)$, and $L_{y\mu}=\cO(\kappa_g)$. Now combining the primal and dual gap functions, we conclude that $\cG_\cZ(x_{k^*},y_{k^*})\leq \cO(\frac{\gamma \kappa_g}{\mu^2}+\frac{\gamma\kappa_g^2}{\mu}+\mu+\frac{1}{\sqrt{\gamma\tau K}}+\sqrt{\frac{\gamma\tau\kappa_g^3}{\mu}})$. Selecting $\mu=\cO(\kappa_g^{3/5}/K^{1/5})$, $\gamma=\frac{\mu^3}{\tau\kappa_g^3}=\cO(1/(\tau \kappa_g^{6/5} K^{3/5}))$, and the primal step-size $\tau$ selected as a constant $\tau\leq \cO(1/\kappa_g^3)$ implies that $\cG_\cZ(x_{k^*},y_{k^*})\leq \cO(\frac{\kappa_g^{3/5}}{K^{1/5}})$. 
Therefore, to achieve $\cG_\cZ(x_{k^*},y_{k^*})\leq \epsilon$, 
Algorithm \ref{i-BRPD:FP} requires at least $K=\cO(\frac{\kappa_g^3}{\epsilon^5})$ 
iterations.
\end{proof}
\begin{corollary}\label{cor:proj-linear-y}
\mma{Suppose the premises of Theorem \ref{thm:proj} holds and we further assume that \mma{$\Phi$ is strongly concave in $y$}. Let $\gamma = \cO(\epsilon^2)$ and \ey{$\tau=\cO(1/\kappa_g^3)$}, then for any $K\geq 1$, there exists $t\in \{0,\hdots,K-1\}$ such that $(x_t,y_t)\in \mathcal{X}\times \mathcal{Y}$ satisfies  $\cG_\cZ(x_t,y_t)\leq \epsilon$ within \mma{$\cO(\frac{\kappa_g^3}{\epsilon^4})$} iterations.}
\end{corollary}
\begin{proof}
    \mma{Suppose that $\Phi(x,\cdot)$ is strongly concave for any $x$ with modulus $\bar{\mu}>0$. This means that we no longer need to add a regularization parameter $\mu$ simplifying the bound on the primal and dual gap functions obtained in Theorem \ref{thm:proj}. In particular, one can replace $\mu$ with $\bar{\mu}$ in the bounds of Lemmas \ref{lem: lower-y-ymu}, \ref{lem: y-ymu-bound}, and \ref{lem: y-bound}. Therefore, the bound on the primal gap function in Theorem \ref{thm:proj} can be rewritten as $\cG_\cX(x_{k^*},y_{k^*})=\frac{1}{\tau} \norm{s_{k^*}-x_{k^*}} 
    \leq  \sqrt{\frac{2(f(x_0) - f(x^*))}{\gamma\tau (K-\bar K)}}+ \sqrt{\frac{2\bL_{\cL_2}\Gamma_1+4\tau M\Gamma_4}{\tau (K-\bar{K})}}+\sqrt{8\tau\gamma  M^2 B_4}$.  
    Next, we restate the dual gap function in \eqref{eq:dual-gap}. In fact for any $k\geq 1$,}
    \mma{
    \begin{align*}
    \mathcal{G}_{\mathcal{Y}} (x_k,y_k) &\leq 
    \frac{2}{\sigma}\rho_D^k \left\| y_0 - y^*_{\bar{\mu}}(x_0) \right\| + \frac{2\gamma L_{y\bar{\mu}}}{\sigma\bar{\mu}(1-\rho_D)}D_\cX + 2L_{y\theta}^\Phi \beta\left(\frac{\rho_D^{k/2}}{1-\beta}+\frac{\beta^{k/2}}{1-\rho_D}\right) \left\|\theta_0 - \theta^*(x_0)\right\|\nonumber\\
    & \quad + \frac{2\gamma \beta L_{y\theta}^\Phi \bL_\theta}{(1-\rho_D)(1-\beta)}D_\cX+ \frac{L_{y\bar{\mu}}}{\bar{\mu}\sigma}\gamma D_\cX + L_{y\theta}^\Phi\beta^k \left\|\theta_0 - \theta^*(x_0)\right\| + \frac{\gamma L_{y\theta}^\Phi\bL_\theta}{1-\beta} D_\cX \nonumber\\
    &=\cO\left(\max\{\rho_D,\beta\}^{k/2}+\gamma\kappa_g^2\right).
\end{align*}}
     \mma{Therefore, $\cG_\cZ(x_{k^*},y_{k^*})\leq \mma{\cO(\gamma\kappa_g^2+\frac{1}{\sqrt{\gamma\tau K}}+\sqrt{\gamma\tau\kappa_g^3})}$, choosing $\gamma=\cO(\kappa_g^{1.5}/K^{1/2})$ and the primal-step-size as constant $\tau=\cO(1/\kappa_g^3)$, it follows that $\cG_\cZ(x_{k^*},y_{k^*})\leq \mma{\cO(\frac{\kappa_g^{3/4}}{K^{1/4}}+\frac{\kappa_g^{3.5}}{K^{1/2}})}$. Therefore, to achieve $\cG_\cZ(x_{k^*},y_{k^*})\leq \epsilon$, Algorithm \ref{i-BRPD:FP} requires $K=\mma{\cO(\frac{\kappa_g^3}{\epsilon^4})}$ iterations.}
\end{proof}
\begin{remark}\label{remak:proj}
    \mma{We would like to highlight that our results in Theorem \ref{thm:proj} demonstrate a new bound on the primal and dual gap function defined in Definition \ref{def:gap-func-Alg2}. As shown in Corollary \ref{cor:proj}, Algorithm \ref{i-BRPD:FP} achieves an $\epsilon$-stationary solution within $\cO(\epsilon^{-5})$ iterations, presenting a novel result under the general setting for nonbilinear SPB problems. This matches the best-known result of $\cO(\epsilon^{-5})$ from \cite{gu2021nonconvex,gu2022min}, which is achieved under the specific problem setup where the upper-level objective function is linear in the maximization component. Furthermore, 
    when the upper-level function $\Phi$ is strongly concave in $y$, our result improves to $\cO(\epsilon^{-4})$ which is the best-known result for the deterministic setting. In the stochastic setting, \cite{hu2022multi} achieves an $\epsilon$-stationary solution within $\cO(\epsilon^{-4})$ iterations for the implicit gradient method, where they used a stronger notion of stationarity compared to ours. To achieve the same $\epsilon$-stationary point, they incur an additional complexity factor of $\cO(\epsilon^{-2})$, resulting in an overall complexity of $\cO(\epsilon^{-6})$ in the stochastic setting -- see Proposition 4.11 in \cite{lin2020gradient}.}
\end{remark}
\begin{remark}
    According to the results of Corollaries \ref{cor:proj-free} and \ref{cor:proj}, we have demonstrated that for solving the general non-bilinear SPB problem, our proposed methods i-BRPD:OPF and i-BRPD:FP require $\mathcal{O}(\epsilon^{-4})$ and $\mathcal{O}(\epsilon^{-5})$ iterations to achieve an $\epsilon$-stationary point, respectively. However, one needs to compare these bounds with extra caution as the metrics used to measure the gap function and the oracle used by the algorithm (i.e., LMO and PO) are fundamentally different and are not directly comparable. 
\end{remark}
\section{Numerical Experiment}\label{sec:numeric}
In this section, we conduct two experiments to evaluate the performance of our algorithms in solving the SPB problem \eyh{by comparing them against the state-of-the-art method MORBiT \cite{gu2022min}.} Specifically, we implement our algorithms to address a robust multi-task linear regression example as described in Section \ref{subsec:examples}. In particular, each task $i\in\{1,\hdots,T\}$ consists of a linear regression model with training and validation datasets denoted by $\cD_i^{tr}$ and $\cD_i^{val}$, respectively. 
For each task $i$, we consider a model in the form of $b_i = A_i(\lambda_iy_i + (1-\lambda_i)x) + \epsilon_i$ \eyh{where $\epsilon_i\sim\cN(0,s)$}, which includes a common coefficient $x\in\reals^d$, task-specific coefficients $y_i\in\reals^d$, and parameter $\lambda_i\in [0,1]$ for a given dataset $\mathcal{D}_i=(A_i,b_i)$ where $A_i\in\mathbb{R}^{n_i\times d}$ and $b_i\in\mathbb{R}^{n_i}$. We define the loss function as $\ell_i(y_i;\mathcal{D}_i)=\frac{1}{2n_i}\norm{A_iy_i-b_i}^2$.

In our experiments, we consider three datasets: the 1979 National Longitudinal Survey of Youth (\textit{NLSY79}) dataset\footnote{\href{https://www.bls.gov/nls/nlsy79.htm}{NLSY79 dataset overview}} ($n=6213, d=21$), the \textit{MTL-dataset} ($n=5642, d=14$), which comprises a mixture of datasets including bodyfat, housing, mg, mpg, and space-ga from the LIBSVM library \cite{chang2011libsvm}, and a synthetic dataset ($n=5000, d=100$), generated randomly from a standard Gaussian distribution, where $n$ represents the number of data points, and $d$ denotes the number of features in each dataset. We distribute the data points evenly among $N=5$ tasks. Then, we allocate 75\% of the data for each task as the training dataset and reserve the remaining 25\% as the validation dataset. To evaluate the performance of our proposed algorithms, we plot the iterates in terms of the gap function defined in Definition \ref{def:gap-func-Alg1}. 

For all algorithms, we let the maximum number of iterations $K=10^4$, and the step-sizes are selected based on their theoretical results. In particular, for i-BRPD:OPF, we let $\gamma=\nu/K^{2/3}$, $\mu = \nu/K^{1/3}$; for i-BRPD:FP, we let $\gamma=\nu/K^{1/2}$, $\mu = \nu/K^{1/4}$, and $\tau = 0.7$; in the both proposed algorithms,  we consider $\alpha = \eta = \frac{2}{\mu_g + L_g}$, and $\sigma = \frac{1}{\mu}$; for MORBiT, we let $\sigma =\nu/K^{3/5}$, $\alpha = 1/K^{2/5}$, and $\tau = \nu/K^{3/5}$, where $\nu>0$ is a tuning parameter that is adjusted to observe the best performance for the considered methods.
\paragraph{Experiment 1.}  \eyh{Considering the worst-task training, we aim to solve the following SPB problem:}
\begin{align}\label{eq:numeric-example1}
    &\min_{\|x\|_1\leq Q, \lambda \in [0,1]^T}\max_{\eta \in \Delta_T} \sum_{i=1}^T \eta_i \ell_i(y_i^*(x,\lambda_i);\cD_i^{val})\\
    &\text{s.t.}~ y_i^*(x,\lambda_i)\in\argmin_{y_i\in\mathbb R^d} \ell_i(\lambda_i y_i + (1-\lambda_i)x;\cD_i^{tr})+\frac{\rho}{2}\norm{y_i}^2,\nonumber
\end{align} 
\eyh{where the constraint $\|x\|_1\leq Q$ is imposed to induce sparsity for some $Q>0$. 
Note that due to the separability of the lower-level problems in $y_i$, \eqref{eq:numeric-example1} can be formulated as a special case of \eqref{eq:main-prob} by considering $(x,\lambda)\in\reals^{d+T}$ as the minimization variables, $\eta\in\reals^T$ as the maximization variable, $\theta=[y_i]_{i=1}^T\in\reals^{dT}$ as the lower-level variable. In this case, $g((x,\lambda),\theta)=\sum_{i=1}^T\ell_i(\lambda_i y_i + (1-\lambda_i)x;\cD_i^{tr}) + \frac{\rho}{2}\norm{\theta}^2$, and $\Phi((x,\lambda),\theta,\eta)=\sum_{i=1}^T \eta_i \ell_i(y_i;\cD_i^{val})$.} It is worth mentioning that, to compare our algorithms with the MORBiT method, we consider a special case of \eqref{eq:main-prob} where the upper-level objective function is assumed to be linear in the maximization component.

The performance of the algorithms is depicted in Figure \ref{fig:mainfig}. Notably, our proposed algorithms exhibit a faster convergence compared to MORBiT. This superiority stems from using an inexact nested approximation technique for estimating the Hessian inverse matrix in our proposed method (see the discussion in Section \ref{sec:main-algorithm}). In contrast, MORBiT approximates the Hessian inverse using the Neumann series which entails a biased approximation requiring multiple Hessian evaluations. Moreover, i-BRPD:OPF algorithm, which is a one-sided projection-free method, demonstrates the fastest convergence behavior among all algorithms, possibly due to allowing for a larger selection of step-size parameters.
\mma{\paragraph{Experiment 2.} We conduct an additional experiment to explore a more general setting in which $\Phi$ is nonlinear in the maximization variable. Specifically, we examine the robust multi-task linear regression problem, similar to the previous experiment but with a different uncertainty set. In particular, consider the following problem:
\begin{align*}
    &\min_{\|x\|_1\leq Q, \lambda \in [0,1]^T}\max_{\eta \in \cP} \sum_{i=1}^T \eta_i \ell_i(y_i^*(x,\lambda_i);\cD_i^{val})\\
    &\text{s.t.}~ y_i^*(x,\lambda_i)\in\argmin_{y_i\in\mathbb R^d} \ell_i(\lambda_i y_i + (1-\lambda_i)x;\cD_i^{tr})+\frac{\rho}{2}\norm{y_i}^2,\nonumber
\end{align*}
where $\cP\triangleq \{\eta \in \Delta_T : V(\eta,\frac{1}{T}\bold{1}_T)\leq \rho' \}$ and $V(\cdot,\cdot)$ denotes the divergence measure between the unknown probability distribution $\eta$ and the uniform distribution $\frac{1}{T}\bold{1}_T$ over the dataset. Using a penalty parameter $\beta>0$, we can relax the divergence constraint and obtain the following problem: 
\begin{align}\label{eq:numeric-example2}
    &\min_{\|x\|_1\leq Q, \lambda \in [0,1]^T}\max_{\eta \in \Delta_T} \sum_{i=1}^T \eta_i \ell_i(y_i^*(x,\lambda_i);\cD_i^{val}) - \frac{\beta}{T} \Big(\frac{1}{2}\|T\eta- \bold{1}_T\|^2 - \rho'\Big)\\
    &\text{s.t.}~ y_i^*(x,\lambda_i)\in\argmin_{y_i\in\mathbb R^d} \ell_i(\lambda_i y_i + (1-\lambda_i)x;\cD_i^{tr})+\frac{\rho}{2}\norm{y_i}^2.\nonumber
\end{align}
Note that in \eqref{eq:numeric-example2} the upper-level function is nonlinear in $\eta$.}\\
\mma{As shown in Figure \ref{fig:expfig2}, our proposed algorithms demonstrate fast convergence when the upper-level objective function is nonlinear in both the primal and dual variables. This underscores the effectiveness of our methods in handling general settings, demonstrating greater versatility compared to existing approaches, which may be restricted to specific problem setups.}

\begin{figure}
    \centering
    \subfloat[NLSY97 dataset]{\includegraphics[width=0.3\textwidth]{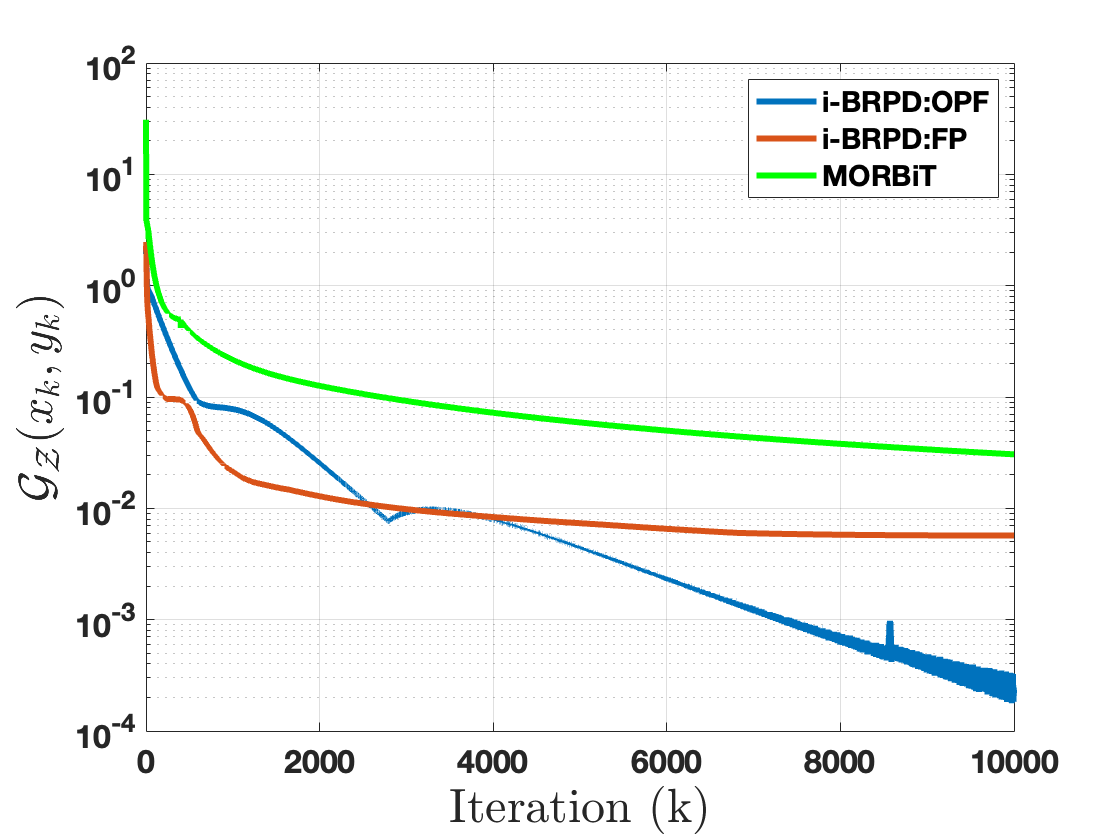}\label{fig:subfig1}}
    \hfill
    \subfloat[MTL dataset]{\includegraphics[width=0.3\textwidth]{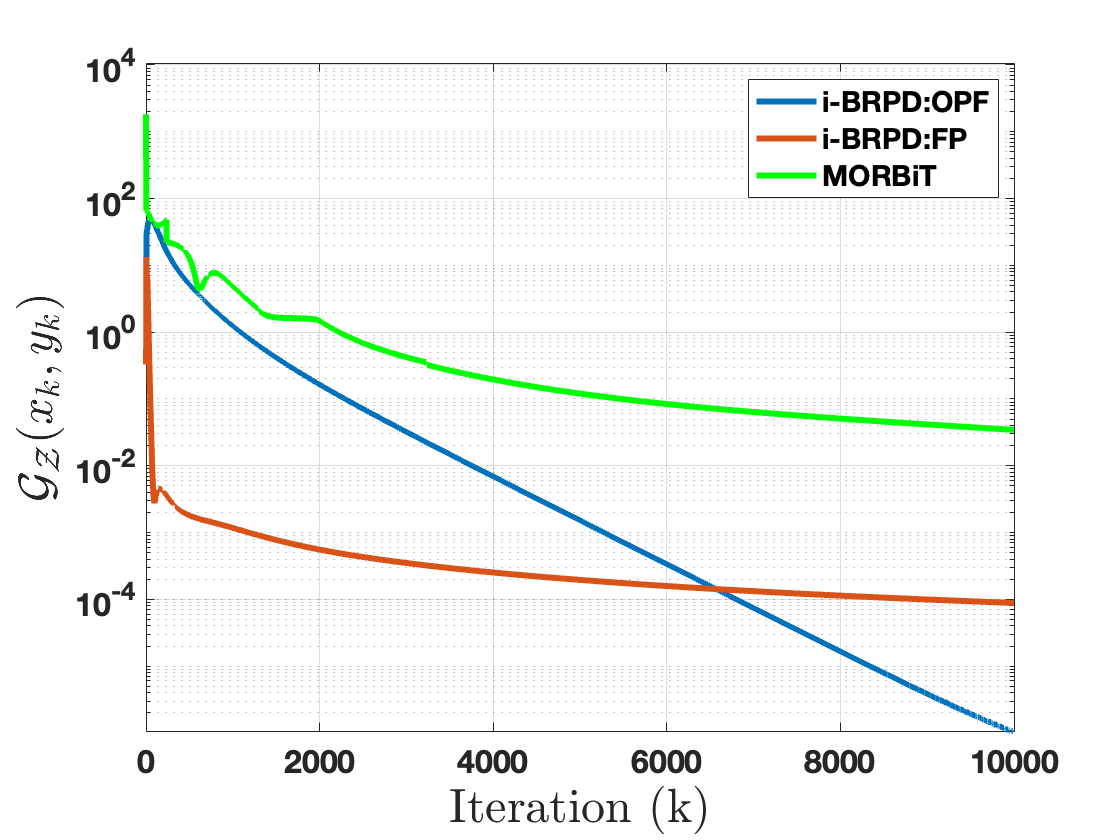}\label{fig:subfig2}}
    \hfill
    \subfloat[Synthetic dataset]{\includegraphics[width=0.3\textwidth]{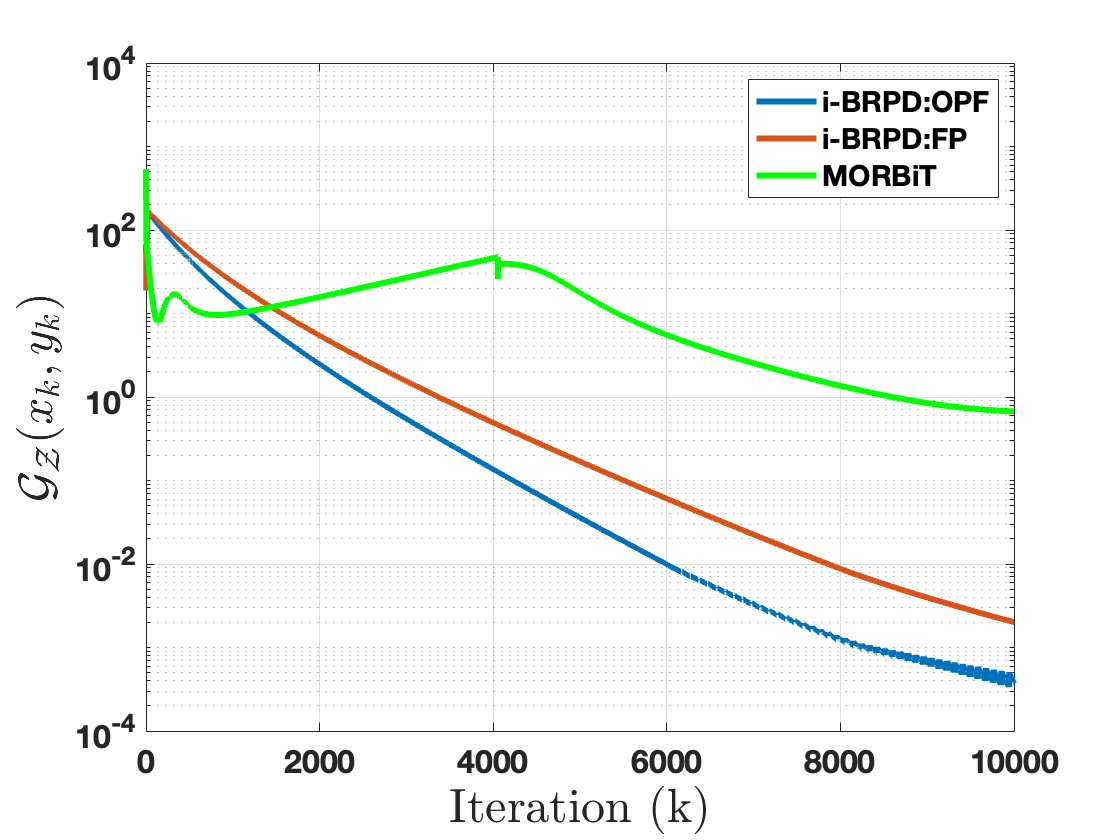}\label{fig:subfig3}}
    \caption{Comparing the performance of our proposed algorithms i-BRPD:OPF (blue) and i-BRPD:FP (red) with MORBiT (green) in Robust Multi-task Linear Regression problem} 
    \label{fig:mainfig}
\end{figure}
\begin{figure}[htp]
    \centering
    \subfloat[NLSY97 dataset]{\includegraphics[width=0.3\textwidth]{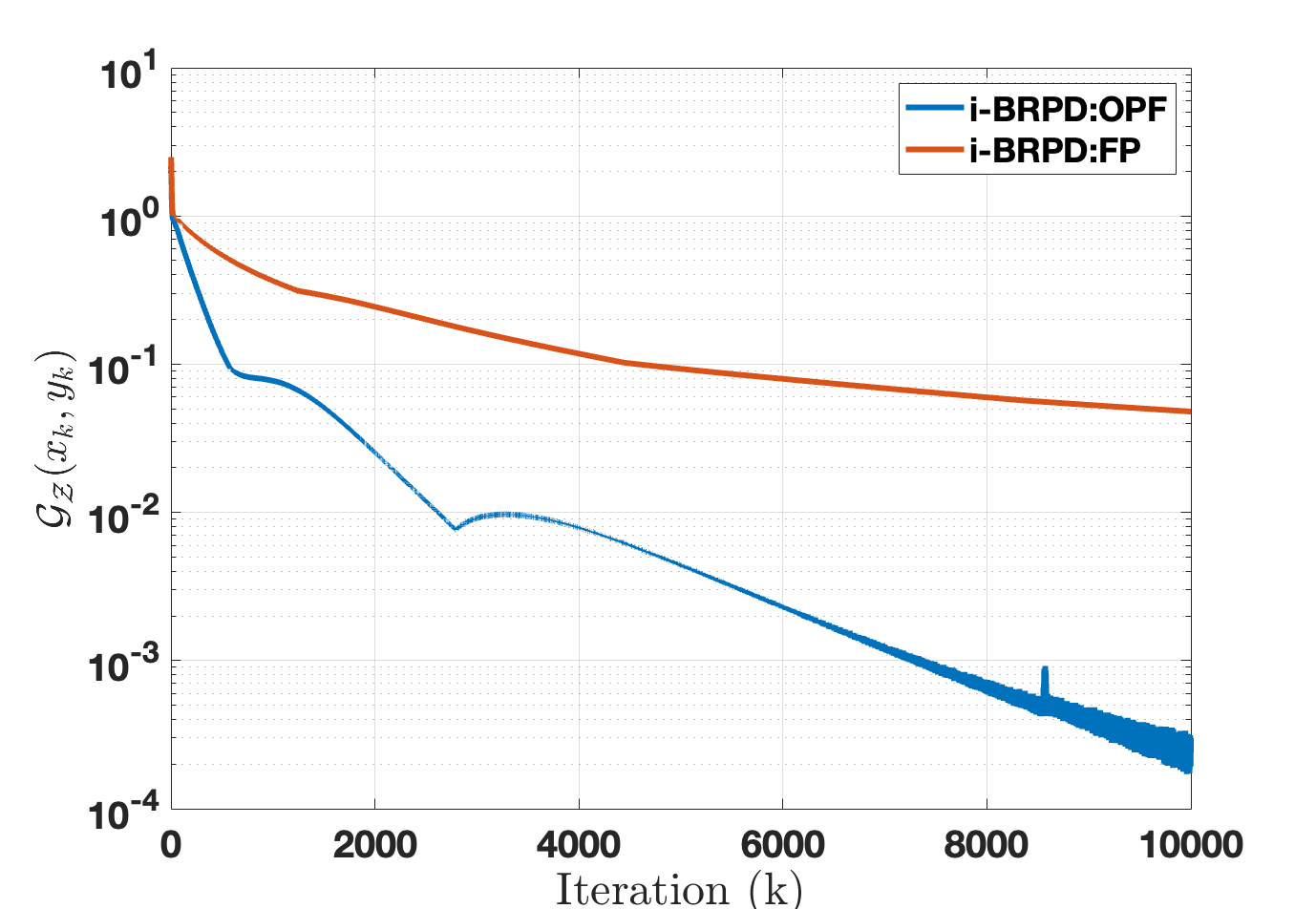}\label{fig:subfig1}}
    \hfill
    \subfloat[MTL dataset]{\includegraphics[width=0.3\textwidth]{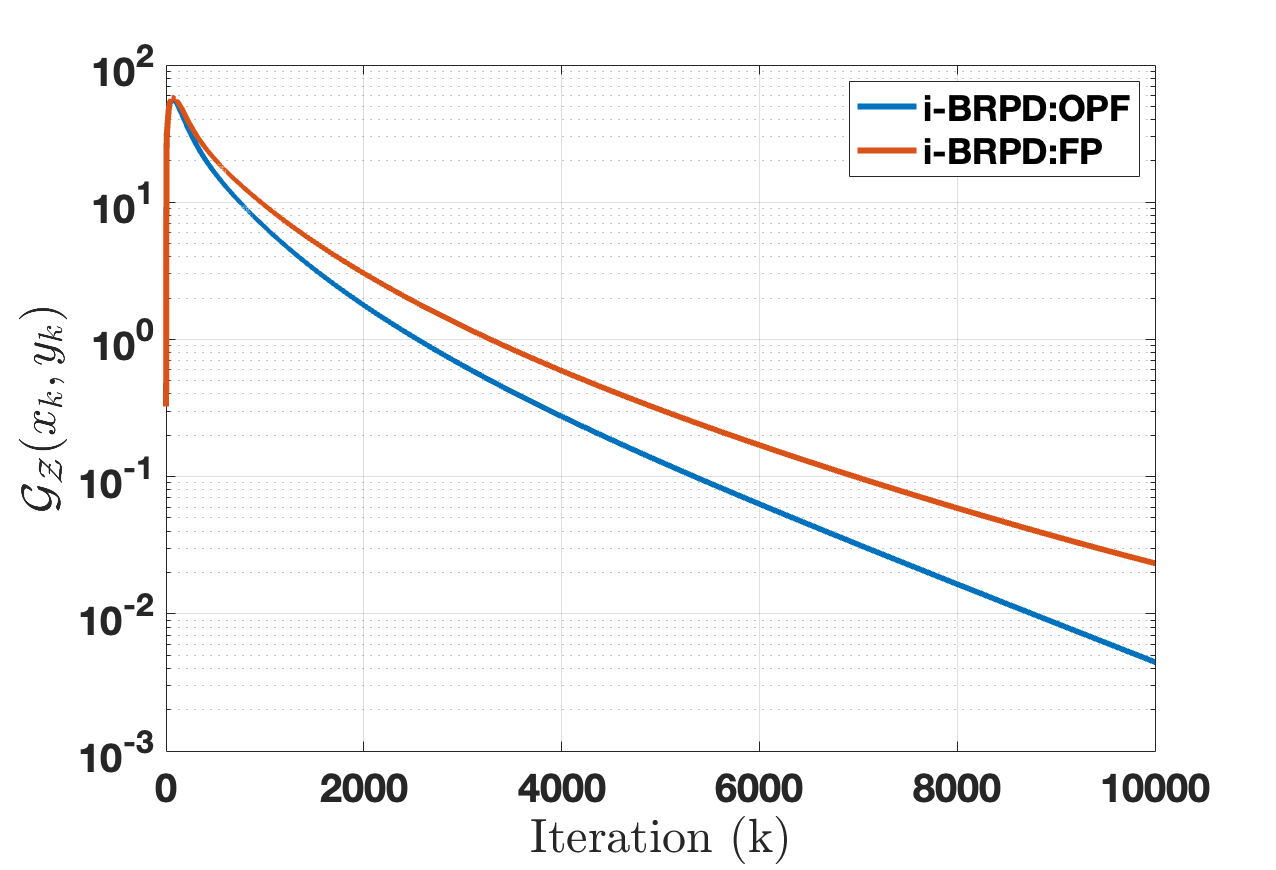}\label{fig:subfig2}}
    \hfill
    \subfloat[Synthetic dataset]{\includegraphics[width=0.3\textwidth]{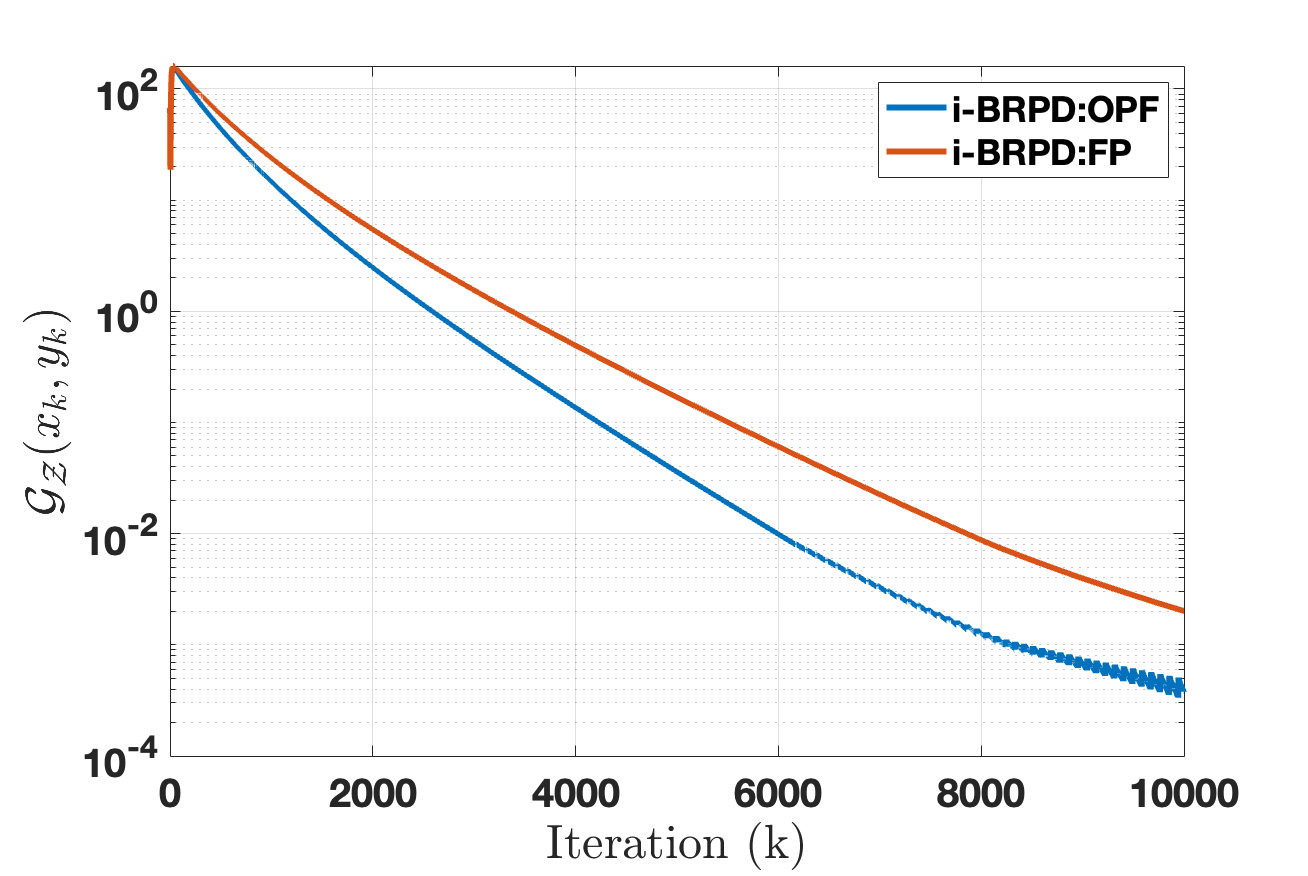}\label{fig:subfig3}}
    \caption{Comparing the performance of our proposed algorithms i-BRPD:OPF (blue) and i-BRPD:FP (red) for solving robust multi-task linear regression problems when the upper-level objective function is nonlinear}
    \label{fig:expfig2}
\end{figure}
\section{{Conclusions}}\label{sec:conclusions}
In this paper, we introduced a class of constrained saddle point problems with a bilevel structure, where the upper-level objective function is nonconvex-concave and smooth over a compact and convex constraint set, subject to a strongly convex lower-level objective function. By leveraging regularization and nested approximation techniques, we initially devised a primal-dual one-sided projection-free algorithm (i-BRPD:OPF), requiring $\mathcal{O}(\epsilon^{-4})$ iterations to attain an $\epsilon$-stationary point, moreover, when the objective function in the upper-level is linear in the maximization component, our results improve to $\cO(\epsilon^{-3})$. Subsequently, we developed a bilevel primal-dual fully projected algorithm (i-BRPD:FP), capable of achieving an $\epsilon$-stationary solution within $\mathcal{O}(\epsilon^{-5})$ iterations. This result improves to $\cO(\epsilon^{-4})$ when the upper-level objective function is strongly concave in $y$. 
Our proposed methods and their convergence guarantee present new results appearing for the first time in the literature for the considered setting, to the best of our knowledge.


\subsubsection*{Acknowledgements}
This work is supported in part by NSF Grant 2127696 and the Arizona Technology and Research Initiative Fund (TRIF) for Innovative Technologies for the Fourth Industrial Revolution. 

\section*{Appendix}\label{sec:appndx}


\section{Auxiliary Lemmas}\label{sec:required-lemmas}
In this section, we provide detailed explanations and proofs for the lemmas supporting the main results of the paper. \\

\begin{lemma}\label{lem:sum-sum-linear-comb}
    Let $\{A_k\}_{k\geq 0}$ be a sequence of nonnegative real numbers, $\rho\in (0,1)$, and $\beta\in (0,1)$. For any $K\geq 1$, 
    \begin{itemize}
        \item[a)] $\sum_{k=0}^{K-1}\sum_{i=0}^k \beta^{k-i} A_i \leq  \frac{1}{1-\beta}\sum_{k=0}^{K-1}A_k$.
        \item[b)] $\sum_{k=0}^{K-1}A_k\sum_{i=0}^{k-1} \beta^{k-i} A_i \leq  \frac{\beta}{1-\beta}\sum_{k=0}^{K-1}A_k^2$.
        \item[c)] $\sum_{k=0}^{K-1} A_k \sum_{i=0}^k \rho^{k-i} \sum_{j=0}^{i-1}\beta^{i-j} A_j \leq \frac{\beta}{(1-\beta)(1-\rho)}\sum_{k=0}^{K-1} A_k^2$. 
    \end{itemize}
\end{lemma}

\begin{proof}
Part (a):
\begin{align*}
    \sum_{k=0}^{K-1}\sum_{i=0}^k \beta^{k-i} A_i &= A_0 \sum_{k=0}^{K-1}\beta^i + \sum_{k=0}^{K-2}\beta^i + \sum_{k=0}^{K-3}\beta^i + ... + \sum_{k=0}^{K-1-(K-1)}\beta^i \nonumber\\
    & \leq \frac{1}{1-\beta}(A_0 + A_1 + ... + A_{K-1})\nonumber\\
    & \leq \frac{1}{1-\beta} \sum_{k=0}^{K-1} A_k.
\end{align*}
Part (b):
Using Young's inequality and part (a) we can show
\begin{align*}
    \sum_{k=0}^{K-1}A_k\sum_{i=0}^{k-1} \beta^{k-i} A_i &\leq \sum_{k=0}^{K-1} \sum_{i=0}^{k-1} \beta^{k-i} (A_k^2 + A_i^2)/2 \nonumber\\
    & \leq \sum_{k=0}^{K-1} \Big(\frac{\beta}{1-\beta}\frac{A_k^2}{2}+ \frac{\beta}{1-\beta}\frac{A_k^2}{2} \Big) \nonumber\\
    & \leq \frac{\beta}{1-\beta} \sum_{k=0}^{K-1} A_k^2.
\end{align*}
To prove part (c), first using Young's inequality we have that 
\begin{align}\label{eq:tiple-sum}
    \sum_{k=0}^{K-1} A_k \sum_{i=0}^k \rho^{k-i} \sum_{j=0}^{i-1}\beta^{i-j} A_j &\leq \sum_{k=0}^{K-1}  \sum_{i=0}^k  \sum_{j=0}^{i-1} \rho^{k-i} \beta^{i-j} (A_k^2+A_j^2)/2 \nonumber \\
    &\leq \sum_{k=0}^{K-1}  \left(\frac{\beta}{(1-\beta)(1-\rho)}\frac{A_k^2}{2} +\sum_{i=0}^k  \sum_{j=0}^{i-1} \rho^{k-i} \beta^{i-j} \frac{A_j^2}{2}\right).
\end{align} 
Now, let us define $A'_i\triangleq \sum_{j=0}^{i-1}\beta^{i-j} A_j^2$, then applying part (a) twice we obtain $\sum_{k=0}^{K-1}\sum_{i=0}^k\rho^{k-i}A'_i \leq \frac{1}{(1-\rho)}\sum_{k=0}^{K-1}A'_k=\frac{1}{(1-\rho)}\sum_{k=0}^{K-1}\sum_{i=0}^{k-1}\beta^{k-i} A_i^2\leq \frac{\beta}{(1-\rho)(1-\beta)}\sum_{k=0}^{K-1}A_k^2$. Therefore, using this inequality within \eqref{eq:tiple-sum} the desired result can be concluded.
\end{proof}

\subsection{Proof of Lemma \ref{lem:lower-v}}\label{sec:proof-lemma-lip-v}

Before presenting the proof of Lemma \ref{lem:theta-gradxL-lip}, we show a preliminary result essential in establishing the result of Lemma \ref{lem:theta-gradxL-lip}. In particular, we show that the map $v(x,y)$ defined in \eqref{eq:v} is Lipschitz continuous. 
\begin{proof}
    We start the proof by recalling that $v(x,y) = [\nabla_{\theta\theta}^2 g(x,\theta^{*}(x))]^{-1} \nabla_\theta\Phi(x,\theta^*(x),y)$. Next, adding and subtracting $[\nabla_{\theta\theta}^2g(x,\theta^*(x))]^{-1}\nabla_\theta \Phi(\overline{x},\theta^*(\overline{x}),\overline{y})$ followed by a triangle inequality leads to,
    \begin{align}\label{eq:lower-v-itr}
    \left\| v(x,y)-v(\overline{x},\overline{y}) \right\| &= \left\| [\nabla_{\theta\theta}^2 g(x,\theta^{*}(x))]^{-1} \nabla_\theta\Phi(x,\theta^*(x),y) \right. \nonumber\\
    & \quad - \left. [\nabla_{\theta\theta}^2 g(\overline{x},\theta^{*}(\overline{x}))]^{-1} \nabla_\theta\Phi(\overline{x},\theta^*(\overline{x}),\overline{y}) \right\| \nonumber\\
    & \leq \left\| [\nabla_{\theta\theta}^2 g(x,\theta^{*}(x))]^{-1} (\nabla_\theta\Phi(x,\theta^*(x),y) - \nabla_\theta\Phi(\overline{x},\theta^*(\overline{x}),\overline{y}) )\right\| \nonumber\\
    & \quad + \left\| ([\nabla_{\theta\theta}^2 g(x,\theta^{*}(x))]^{-1} -[\nabla_{\theta\theta}^2 g(\overline{x},\theta^{*}(\overline{x}))]^{-1})\nabla_\theta\Phi(\overline{x},\theta^*(\overline{x}),\overline{y})\right\| \nonumber\\
    & \leq \frac{1}{\mu_g}(L_{\theta x}^\Phi \left\| x - \overline{x} \right\| + L_{\theta \theta}^\Phi \left\| \theta^*(x) - \theta^*(\overline{x}) \right\| + L_{\theta y}^\Phi \left\|\overline{y}-y \right\|) \nonumber\\
    & \quad + C_\theta^\Phi \left\| [\nabla_{\theta\theta}^2 g(x,\theta^{*}(x))]^{-1} - [\nabla_{\theta\theta}^2 g(\overline{x},\theta^{*}(\overline{x}))]^{-1} \right\|,
    \end{align}
    where in the last inequality, we used Assumptions \ref{assump:grad-xytheta-lip} and \ref{assump:g-conditions}-(3) along with the premises of Assumption \ref{assump:grad-phi-bounded}. Moreover, for any invertible matrices $H_1$ and $H_2$, we have that
    \begin{align}\label{eq:inv-H-bound}
    \left\| H_2^{-1} - H_1^{-1} \right\| = \left\| H_1^{-1}(H_1-H_2)H_2^{-1} \right\|\leq \left\| H_1^{-1} \right\|\left\| H_2^{-1} \right\|\left\| H_1-H_2 \right\|.
    \end{align}
    Therefore, using the result of Lemma \ref{lem:theta-gradxL-lip}-(I) and \eqref{eq:inv-H-bound}, we can further bound inequality \eqref{eq:lower-v-itr} as follows,
    \begin{align*}
        \left\| v(x,y)-v(\overline{x},\overline{y}) \right\| &\leq \frac{1}{\mu_g}(L_{\theta x}^\Phi \left\| x - \overline{x} \right\| + L_{\theta \theta}^\Phi \bL_{\theta} \left\| x - \overline{x} \right\| + L_{\theta y}^\Phi \left\|\overline{y}-y \right\|) \nonumber\\
        & \quad + C_\theta^\Phi \left\| [\nabla_{\theta\theta}^2 g(x,\theta^{*}(x))]^{-1} - [\nabla_{\theta\theta}^2 g(\overline{x},\theta^{*}(\overline{x}))]^{-1} \right\| \nonumber\\
        & \leq \frac{1}{\mu_g} (L_{\theta x}^\Phi + L_{\theta \theta}^\Phi\bL_\theta)\left\| x - \overline{x} \right\| + \frac{1}{\mu_g}L_{\theta y}^\Phi \left\|\overline{y}-y \right\| \nonumber\\
        & \quad + \frac{C_\theta^\Phi}{\mu^2_g} L_{\theta\theta}^g (\left\| x - \overline{x} \right\| + \left\| \theta^*(x) - \theta^*(\overline{x}) \right\|) \nonumber\\
        & \leq (\frac{L_{\theta x}^\Phi + L_{\theta \theta}^\Phi\bL_\theta}{\mu_g} + \frac{C_\theta^\Phi L_{\theta\theta}^g}{\mu^2_g} (1+\bL_\theta))\left\| x - \overline{x} \right\| + \frac{1}{\mu_g}L_{\theta y}^\Phi \left\|\overline{y}-y \right\|.
    \end{align*}
    The result follows from the definitions of $\bC_{v1}$ and $\bC_{v2}$.
\end{proof}
\subsection{Proof of Lemma \ref{lem:theta-gradxL-lip}}\label{sec:proof-lemma-basic}
\textbf{Part (I):} Recall that $\theta^*(x)$ is the minimizer of the lower-level problem whose objective function is strongly convex, therefore
\begin{align*}
\mu_g \left\| \theta^*(x) - \theta^*(\overline{x}) \right\|^2 &\leq \langle \nabla_\theta g(x,\theta^*(x)) - \nabla_\theta g(x,\theta^*(\overline{x})),\theta^*(x) - \theta^*(\overline{x}) \rangle \\
& = \langle \nabla_\theta g(\overline{x},\theta^*(\overline{x})) - \nabla_\theta g(x,\theta^*(\overline{x})),\theta^*(x) - \theta^*(\overline{x}) \rangle .
\end{align*}
Note that $\nabla_\theta g(x,\theta^*(x)) = \nabla_\theta g(\overline{x},\theta^*(\overline{x})) = 0$. Using the Cauchy-Schwartz inequality, we have
\begin{align*}
\mu_g \left\| \theta^*(x) - \theta^*(\overline{x}) \right\|^2 &\leq \left\| \nabla_\theta g(\overline{x},\theta^*(\overline{x})) - \nabla_\theta g(x,\theta^*(\overline{x}))\right\| \left\|\theta^*(x) - \theta^*(\overline{x}) \right\| \\
& \leq C_{\theta x}^g \left\| x - \overline{x} \right\| \left\| \theta^*(x) - \theta^*(\overline{x}) \right\|, 
\end{align*}
where the last inequality is obtained by using Assumption \ref{assump:g-conditions}. Thus, we conclude that $\mu_g \left\| \theta^*(x) - \theta^*(\overline{x}) \right\| \leq C_{\theta x}^g \left\| x - \overline{x} \right\|$ which leads to the desired result in part (I).

\textbf{Part (II):} We start proving this part using the definition of $\nabla_x \mathcal{L}(x,y)$ stated in \eqref{eq:grad-x-L}. Using the triangle inequality, we obtain,
\begin{align*}
    \left\| \nabla_x\mathcal{L}(x,y) - \nabla_x\mathcal{L}(\overline{x},\overline{y}) \right\| &= \left\| \nabla_x \Phi(x,\theta^*(x),y) -\nabla_{\theta x}^2 g(x,\theta^{*}(x)) v(x,y) \right. \nonumber\\
    & \quad- \left. (\nabla_x \Phi(\overline{x},\theta^*(\overline{x}),\overline{y}) -\nabla_{\theta x}^2 g(\overline{x},\theta^{*}(\overline{x})) v(\overline{x},\overline{y})\right\| \nonumber\\
    & \leq \left\| \nabla_x \Phi(x,\theta^*(x),y) - \nabla_x \Phi(\overline{x},\theta^*(\overline{x}),\overline{y}) \right\| \nonumber\\
    & \quad +\left\| [\nabla_{\theta x}^2 g(\overline{x},\theta^{*}(\overline{x})) v(\overline{x},\overline{y}) - \nabla_{\theta x}^2 g(\overline{x},\theta^{*}(\overline{x})) v(x,y)] \right. \nonumber\\
    & \quad + \left. [\nabla_{\theta x}^2 g(\overline{x},\theta^{*}(\overline{x})) v(x,y) - \nabla_{\theta x}^2 g(x,\theta^{*}(x)) v(x,y)] \right\|,
\end{align*}
where the second term of the RHS follows from adding and subtracting the term $\nabla_{\theta x}^2 g(\overline{x},\theta^{*}(\overline{x})) v(x,y)$.\\
Next, from Assumptions \ref{assump:grad-xytheta-lip}-(1) and \ref{assump:g-conditions}-(2) together with the triangle inequality, we conclude that
\begin{align*}
    \left\| \nabla_x\mathcal{L}(x,y) - \nabla_x\mathcal{L}(\overline{x},\overline{y}) \right\| &\leq L_{xx}^\Phi \left\| x - \overline{x} \right\| + L_{x\theta}^\Phi \left\| \theta^*(x) - \theta^*(\overline{x})\right\| + L_{xy}^\Phi \left\|\overline{y}-y \right\| \nonumber\\
    & \quad + C_{\theta x}^g \left\| v(\overline{x},\overline{y}) - v(x,y) \right\| + \frac{C_\theta^\Phi}{\mu_g}  \left\|\nabla_{\theta x}^2 g(\overline{x},\theta^{*}(\overline{x}))- \nabla_{\theta x}^2 g(x,\theta^{*}(x)) \right\|.
\end{align*}
Note that in the last inequality, we use the fact that $\left\| v(x,y) \right\| = \left\| [\nabla_{\theta\theta}^2 g(x,\theta^{*}(x))]^{-1} \nabla_\theta\Phi(x,\theta^*(x),y)  \right\| \leq \frac{C_\theta^\Phi}{\mu_g}$.\\
Combining the result of Lemma \ref{lem:theta-gradxL-lip}-(I) and Assumption \ref{assump:grad-phi-bounded} with the Assumption \ref{assump:g-conditions}-(4) leads to
\begin{align*}
    \left\| \nabla_x\mathcal{L}(x,y) - \nabla_x\mathcal{L}(\overline{x},\overline{y}) \right\| &\leq L_{xx}^\Phi \left\| x - \overline{x} \right\| + L_{x\theta}^\Phi \bL_\theta \left\| x - \overline{x}\right\| + L_{xy}^\Phi \left\|\overline{y}-y \right\| \nonumber\\
    & \quad + C_{\theta x}^g (\bC_{v1}\left\| x -\overline{x} \right\| + \bC_{v2}\left\| y -\overline{y} \right\|) \nonumber\\
    & \quad + \frac{C_\theta^\Phi}{\mu_g} L_{\theta x}^g (\left\|\ x - \overline{x}\right\| + \left\|\ \theta^*(x) - \theta^*(\overline{x})\right\|) \nonumber\\
    & \leq L_{xx}^\Phi \left\| x - \overline{x} \right\| + L_{x\theta}^\Phi \bL_\theta \left\| x - \overline{x}\right\| + L_{xy}^\Phi \left\|\overline{y}-y \right\| \nonumber\\
    & \quad + C_{\theta x}^g (\bC_{v1}\left\| x -\overline{x} \right\| + \bC_{v2}\left\| y -\overline{y} \right\|) \nonumber\\
    & \quad + \frac{C_\theta^\Phi}{\mu_g} L_{\theta x}^g (\left\|\ x - \overline{x}\right\| + \bL_\theta \left\| x - \overline{x}\right\|) \nonumber\\
    & \leq \Big(L_{xx}^\Phi  + L_{x\theta}^\Phi \bL_\theta + C_{\theta x}^g \bC_{v1} + \frac{C_\theta^\Phi}{\mu_g} L_{\theta x}^g (1 + \bL_\theta)\Big) \left\| x - \overline{x}\right\| \nonumber\\
    & \quad + (L_{xy}^\Phi + C_{\theta x}^g \bC_{v2})\left\| y -\overline{y} \right\|,
\end{align*}
where $\bL_{\mathcal L_1} = (L_{xx}^\Phi  + L_{x\theta}^\Phi \bL_\theta + C_{\theta x}^g \bC_{v1} + \frac{C_\theta^\Phi}{\mu_g} L_{\theta x}^g (1 + \bL_\theta))$ and $\bL_{\mathcal L_2} = (L_{xy}^\Phi + C_{\theta x}^g \bC_{v2})$.

\textbf{Part (III):} Using the definition of $\nabla_y\mathcal{L}(x,y)$ in \eqref{eq:grady-L} and Assumption \ref{assump:grad-xytheta-lip}-(3), we can show
\begin{align*}
    \|\nabla_y \mathcal{L}(x,y) - \nabla_y \mathcal{L}(\overline{x},\overline{y}) \| & = \|\nabla_y \Phi(x,\theta^*(x),y)-\nabla_y \Phi(\overline{x},\theta^*(x),\overline{y})\| \nonumber\\
    & \leq L_{yx}^\Phi \left\| x - \overline{x} \right\| + L_{y\theta}^\Phi \|\theta^*(x) - \theta^*(\overline{x})\|+ L_{yy}^\Phi \left\|y-\overline{y} \right\| \nonumber\\
    & \leq L_{yx}^\Phi \left\| x - \overline{x} \right\| + L_{y\theta}^\Phi \bL_\theta \left\| x - \overline{x} \right\| + L_{yy}^\Phi \left\|y-\overline{y} \right\| \nonumber\\
    & \leq (L_{yx}^\Phi + L_{y\theta}^\Phi \bL_\theta) \left\| x - \overline{x} \right\| + L_{yy}^\Phi \left\|y-\overline{y} \right\|,
\end{align*}
where the last inequality is obtained by using Lemma \ref{lem:theta-gradxL-lip}-(I).

\bibliographystyle{siam}
\bibliography{reference}
\end{document}